\documentclass[12pt]{article}
\usepackage{graphicx} % Required for inserting images
\usepackage[a4paper, total={6.5in, 9.5in}]{geometry}
\usepackage{amsmath,amssymb, amsthm}
\usepackage[normalem]{ulem}
\usepackage{natbib}
\usepackage{mathtools}
\usepackage{algorithm} 
\usepackage{algpseudocode}
\usepackage{color}
\newtheorem{theorem}{Theorem}[section]

\newtheorem{remark}{Remark}[section]
\usepackage{hyperref}
\usepackage[capitalise]{cleveref}
\usepackage{comment}
\usepackage{bbm}
\usepackage{authblk}

\DeclarePairedDelimiterX{\Xnorm}[1]{\bigg\lVert}{\bigg\rVert}{#1}
\newcommand\norm[1]{\big\|#1\big\|}

\newtheorem{lemma}{Lemma}[section]
\newcommand*{\E}{\mathbb{E}}

\newcommand*{\p}{\mathbb{P}}

\newcommand*{\N}{\mathbb{N}}

\definecolor{darkblue}{rgb}{.1, 0.1,.8}
\definecolor{darkgreen}{rgb}{0,0.8,0.2}
\definecolor{darkred}{rgb}{.8, .1,.1}
\definecolor{violet}{RGB}{148,0,211}

\DeclareMathOperator*{\argmax}{arg\,max}

\newcommand\floor[1]{\left\lfloor #1 \right\rfloor}
\newcommand{\indicator}[1]{\mathbf{1}_{\{#1\}}}

\title{Multiple change point detection in functional data with applications to biomechanical fatigue data}

\author[1]{Patrick Bastian}
\author[2]{Rupsa Basu}
\author[1]{Holger Dette}
\affil[1]{Ruhr-Universität Bochum, Germany}
\affil[2]{University of Twente., The Netherlands  and Universität zu Köln, Germany}
\begin{document}

%\title{Multiple change point detection in functional data with applications to biomechanical fatigue data}

\maketitle

%\date{\today}

\begin{abstract}

  Injuries to the lower extremity joints are often debilitating, particularly for professional athletes. Understanding the onset of stressful conditions on these joints is therefore important in order to ensure prevention of injuries as well as individualised training for enhanced athletic performance. We study the biomechanical joint angles from the hip, knee and ankle for runners who are experiencing fatigue. The data is cyclic in nature and densely collected by body worn sensors, which makes it ideal to work with in the functional data analysis (FDA) framework.

We develop a new method for multiple change point detection for functional data, which improves the state of the art 
with respect to at least two  novel aspects. First, the curves are compared with respect to their maximum  absolute deviation, which leads to a better interpretation of local changes in the functional data compared to classical $L^2$-approaches. Secondly, as slight aberrations are to be often expected in a human movement data, our method will not detect arbitrarily small changes but hunts for relevant changes, where maximum absolute deviation between the curves exceeds a specified threshold, say $\Delta >0$.
We recover multiple changes in a long functional time series of biomechanical knee angle data, which are larger than the desired threshold $\Delta$, allowing us to identify changes purely due to fatigue. In this work, we analyse data from both controlled indoor as well as from an uncontrolled outdoor (marathon) setting.
\end{abstract}

\section{Introduction}
\label{sec1} 
  \def\theequation{1.\arabic{equation}}	
\setcounter{equation}{0}

In this paper we develop a novel methodology for
detecting (multiple) change points  in functional  time series. Our work is motivated by data analysis in the field of human movement, a highly relevant subject area focused on comprehending the impact of disease and aging on physical ability as well as enhancing performance in sports. Human movement is often characterized by its continuous and repetitive motions, which in conjunction with dense data collection via sensors
%, as in the case of biomechanical joint angle data,  
naturally leads to a functional data framework. In Figure  \ref{fig:runner_n_angles} we display typical data  for the different lower extremity joints  from one runner. The left part of the figure shows  strides (cycles) corresponding to angular data measured at the hip, knee and ankles between consecutive contacts of the foot with the ground. The  resulting    functional time series of the different strides  of the runner are shown in the right part of Figure \ref{fig:runner_n_angles}. Within human movement analysis, it is particularly important to detect problematic movements under stress conditions like fatigue and to understand the adjustments made by the body in this case.
In order to do that, it is essential to detect \textit{when} fatigue has kicked-in and the 
 onset of these stressors are characterised by deviations in this type of data. However, this problem is a difficult one, as such subtle  changes, often characterised by spikes within the functions,  are not very visible in plots of consecutive strides  as displayed in  Figure \ref{fig:runner_n_angles}, respectively.
 Therefore, detecting changes in running data  and inferring that when a change occurs, it is indeed  due to fatigue is a very challenging task  amounting to looking for subtle yet significant changes in a long functional time series.

%\vspace{-.25cm}
\begin{figure}[t]
    \centering
    \includegraphics[scale= 0.4]{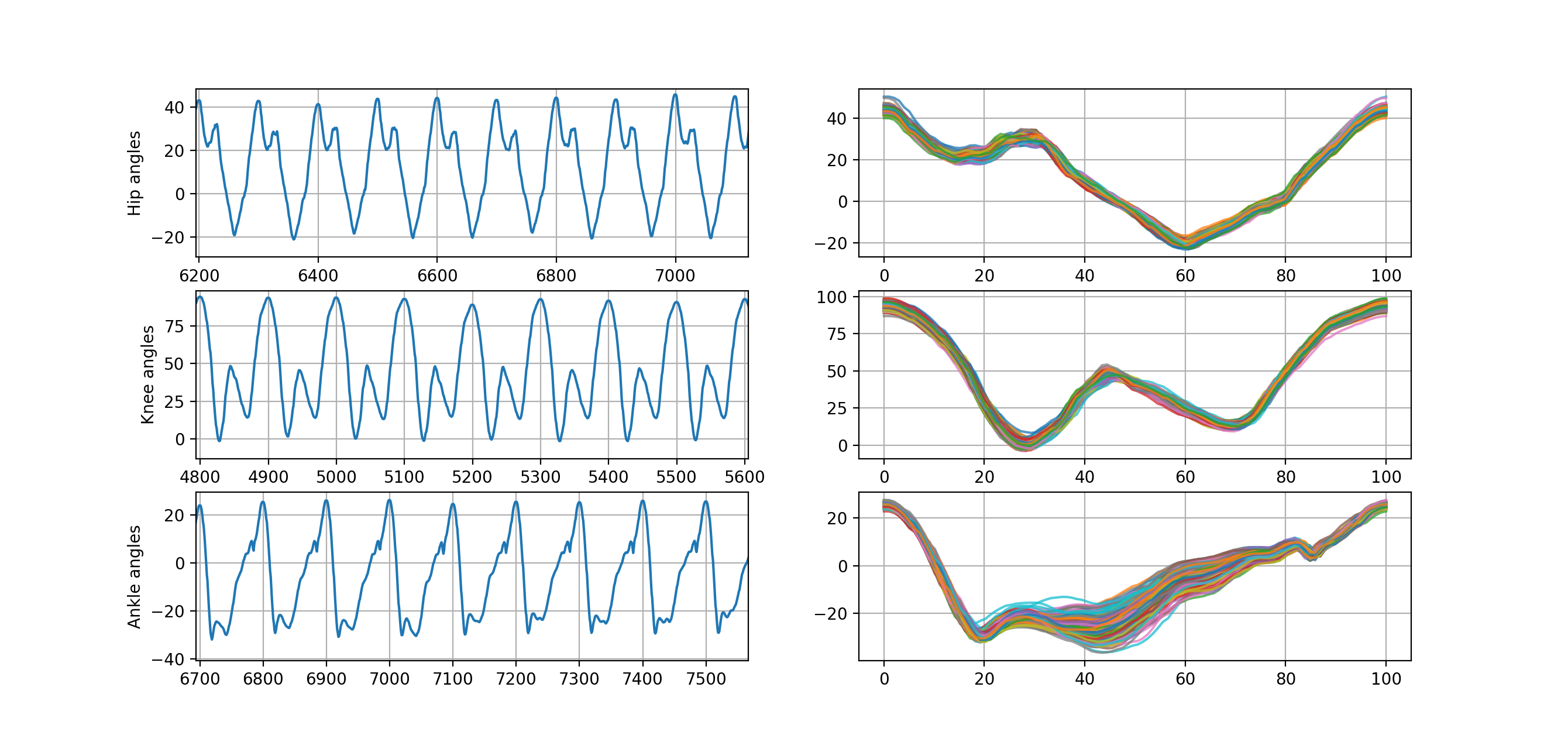}
    %\vspace{-.5cm}
    \caption{\it Biomechanical angle data from the hip (top), knee (middle), ankle (bottom)  for a single runner. Left panes:  Consecutive strides of a runner from a part of the run. Right panels:  Resulting  functional time series of the different strides.
     }
    \label{fig:runner_n_angles}
\end{figure}

Mathematically, this means that we are supremely concerned with locating multiple change points in a functional time series as it is to be expected that the adaptation of the body cannot be a one-time event.  In this regard, it is of importance to note that very often 
there appear ``small'' deviations in  movement data which are not caused by fatigue, but 
due to measurement errors and other environmental aspects during data collection. A typical example 
causing such  small changes in the functional time series is the shifting of sensors due to sweating.  Although, mathematically, the time points corresponding to these deviations could be considered  as change points, they are not of practical interest because the deviations at these points are  very small  as compared to the adaptations made by the body under fatigue. More precisely, in the examples under consideration, changes caused by  environmental aspects usually cause only  minor deviations, while changes due to fatigue lead  to larger deviations.
Thus statistical methodology is required which is able to detect only the ``scientifically relevant'' change points, that is the locations corresponding to changes due to fatigue. We therefore look at deviations (or spikes) in the functional data that are larger than a certain threshold. 
%This threshold value  therefore eliminates the changes due to external factors. 

\begin{figure}
    \centering
    \includegraphics[scale=0.15]{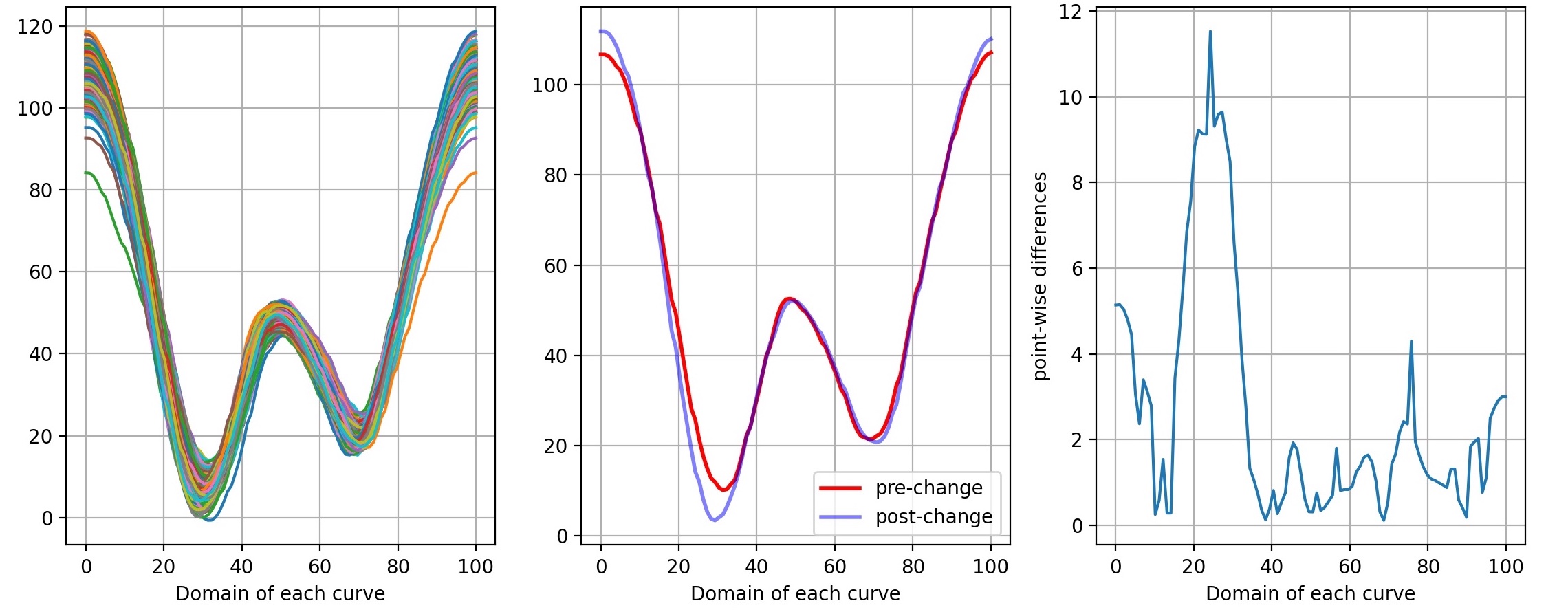}    
    \caption{\it Left:  knee angle data from each cycle over the course of the run. Middle:
mean curves, prior to- and after- the change point for a single runner. Right:   absolute  difference between the mean curves.}   
    \label{motiv_fig}
\end{figure}

This point of view raises the question how to measure deviations between two observations of the functional time series. Most of the literature has dealt with developing  $L^2$-space based methodology  for which there exists by now a fully fledged theory.  If one is interested in arbitrary small deviations  the choice of the norm does not matter (because the norm of the difference between two curves vanishes if and only if the difference vanishes as well). However, if one is  interested in 
differences  with a norm exceeding a certain threshold, the choice of the norm is a more delicate problem. We argue that in the present context of human movement data it is more reasonable to measure deviations with respect to the 
maximum deviation of the curves because of the sensitivity of the sup-norm to  local spikes. A typical situation  is displayed in  Figure \ref{motiv_fig}.
The middle panel shows the estimated mean of the curves 
of the functional time series (displayed in the left panel) before and after a relevant change point, which has been estimated by the procedure to be developed in this paper. In the right part  we plot the absolute difference of the two mean curves, and we observe 
one rather large spike. Such  local spikes are typical for this type of  data and may be easier to detect  using the sup-norm  instead of the $L^2$-norm. Therefore, by using the sup-norm, we capture local behaviour within the functional data, which may come in the form of short but significant spikes. This is particularly relevant for the biomechanical data under consideration given that human movement will not change dramatically at the stride level (in terms of shape, which would be captured well by $L^2$) but locally within the strides and may therefore  only   be detectable by using the sup-norm.   In this case  the change point problem  becomes much more challenging as we are abandoning the Hilbert-space setting in favor of a Banach-space framework, for which the theory is not so well developed.

\medskip

{\bf Our contribution } 
In this paper  we will develop  a new approach for detecting multiple change points in functional time series that distinguishes itself from the currently  available methodology in various aspects. First, since all functions 
considered in our applications are at least continuous and probably   smoother 
than that, the new change point detector will be tailored to applications for functional data in the space of continuous functions $C([0,1])$ where deviations are measured by the sup-norm. Second, compared to the literature, the procedure proposed in this paper detects relevant change points, which means deviations (measured by the sup-norm) exceeding a certain (positive) threshold. On the  one hand  this addresses  the practical demands in the human movement data where one is only interested in changes caused by fatigue. On the other hand,
looking only at relevant changes reflects the fact that in many applications it is not plausible 
that two functions of a time series  are exactly the same (even if we compare their expectations).  
This point of view is in line with 
 \cite{tukey1991}, who  argued in the context of multiple comparisons of means that  
{\it  $\ldots $
``All we know about the world teaches us that the
 effects of A and B are always different  - in some 
 decimal place - for any A and B.  Thus asking ``Are
 the effects different?'' is foolish''.  $\ldots $.}

Third, giving up on the theoretically convenient $L^2$-space 
and entering the Banach-space setting means that substantially more effort is necessary to establish the validity of our approach including  statistical guarantees.  This applies even more  when one also aims for detecting  only the  relevant change points as we leave methodologically convenient stationarity assumptions.  Fourth, to our knowledge, this is the first time that fatigue detection for lower extremity joint angles is carried out using the framework of change point detection for functional data analysis.
In particular,  the above breakthroughs allow us to study this type of data in a far more coherent manner, permitting deductions  on the characteristics,  which may not have been possible based on other methods. Fourth, such an analysis using advanced statistical methodology is a first of its kind for studying fatigue detection in biomechanical lower extremity joint angle data. While experts in biomechanical sciences have known that fatigue brings about pronounced changes in the way we move, it was so far unclear how this change actual impacts the movements of the joints, in particular the knee joint. Such a methodology allows us to study the individualised characteristics of each runner under stressful, prolonged running conditions. The novel aspect of this methodology is the dual contribution of relevant changes of size $\Delta$ as well as that of working with fully functional methodology in the supnorm.  The former, allows us to pick up changes which are larger than a certain size $\Delta$ while the latter enables us to identify the exact part of a cyclic movement that is changing under fatigue. This contributes majorly to our understanding how our body copes with fatigue as well as the straining effects of enduring such a coping mechanisms ultimately resulting in muscle wear and tear as well as injuries.

\medskip

\textbf{Related literature }  Change point analysis of time series data is by now a well studied field with an enormous amount of publications by numerous authors. 
%It pertains towards detecting if a certain data set changes significantly at a given time point. Further, with the advent of data sets observed over  longer time periods, the requirement for detecting multiple change points  arise.
Meanwhile, the literature on addressing the problem of detecting  multiple structural breaks is also extensive and exemplary we mention
(wild) binary segmentation 
\citep{fryzlewicz2014wild},  
pruned exact linear time tests  \citep{killick2012optimal,maidstoneetal2017}, moving sum scan \citep{eichingerkirch2020}
and simultaneous multiscale change point estimators 
\citep{frick2014multiscale,li2019multiscale,dette2020multiscale},
narrowest over threshold \citep{baranowskietal2019} and the references in these works. The  optimality  of univariate multiple change point detection procedures has recently been investigated by \cite{wang2020univariate}.   
Extensions have been made on multiple change detection for multivariate and high dimensional data;  see \cite{chofry2015,cho2016,wangtengyao,padilla2021optimal,kovacs2023seeded} among many others.   

For functional data, change point analysis has mainly been carried out under the assumption of one change point \citep[see, for example,][]{aue2009estimation,berkesetal2009,astonkrich,horvath2014testing,shapirov2016,bucchia2017change,Aue2017,STOEHR2021}. Methodology for detecting one relevant change point has been developed by \cite{dette2020functional,dettekokot} and   \citep{dette2021detecting} among others. 
Most of these work refers to $L^2$-space methodology \citep[see][for an exception]{dette2020functional}. 

The literature on detecting multiple change points in functional time series is more scarce and - to our best knowledge - we are only aware of the work \cite{chiouetal2019,harrisetal2022,rice2022consistency} and  \cite{Madrid2022}. The first three references are mostly concerned with segmentation methods for densely observed data while \cite{Madrid2022} also consider sparsely observed functional data, in particular none of them consider relevant change points nor do they measure the deviation between the curves by the sup-norm as we do.

\section{Relevant change points in functional data }
\label{sec2} 
  \def\theequation{2.\arabic{equation}}	
\setcounter{equation}{0}

Let  $C ( T)$ denote the space of continuous functions defined on the set $T$. Throughout this paper $T$ is either the interval $T=[0,1]$  or a rectangular of the form  $T=[u,v] \times [0,1]$  with $0 \leq u  < v \leq 1$.
We define by $\|  \cdot \|_\infty $ the sup-norm on $C(T)$ (the corresponding space will always be clear from the context), and denote for 
   $f \in C([u,v] \times [0,1])$  by 
   $$
   \| f (s, \cdot )\|_\infty  = \max_{t \in [0,1] } |f (s, t) | 
   $$
   the sup-norm of the function $ t \to  f(s, t )$ (for fixed $s$). 
   We
consider a triangular-array $\{ X_{n, j} |  \, j = 1, \dots, n \}_ {n \in \mathbb{N}}$ of $C ( [0,1])$-valued random variables 
with the representation
\begin{align}\label{eqn:model}
    X_{n, j} = \mu_{n, j} + \epsilon_{n,j},  \quad \quad j = 1, \dots, n,  
\end{align}
where $\{\epsilon_{n,j} \}_{j=1, \ldots, n} $  is a row-wise centred stationary process 
and $\mu_{n, j} = \E [X_{n, j}] \in C ( [0,1])$ denotes the mean function of  $X_{n, j}$. 
For  the sake of a simple notation,  we will  suppress 
in this section the dependence of the functions on the sample size $n \in \N$  and 
use the notations 
$X_j, \epsilon_j$  instead of
 $X_{n,j},  \epsilon_{n,j}$.
Further, 
assumptions regarding the probabilistic structure of the random elements will be stated below in Section \ref{sec4}, where we discuss the theoretical properties of the proposed procedure. 

We assume that there exists  a  set 
$$
S = \{s_1 , s_2, \dots, s_m\}\subset (0,1)
$$
of change points in the sequence of  mean functions with $s_1 \leq  s_2 \leq   \dots \leq s_m $, where $m$ is finite in the sense that it does not depend on $n$.  More precisely,  for any $1 \leq i \leq m <\infty$ we assume 
\begin{align*}
   \mu_i:= \mu_{\floor{ns_{i-1}}+1}=\mu_{\floor{ns_{i-1}}+2}=...= \mu_{\floor{ns_{i}}} \neq \mu_{i+1}
\end{align*}
where we adopt the convention $s_0=0$ and $s_{m+1}=1$. 
For example,  $\mu_1$ is the mean function before the first  change point   $\floor{ns_{1}}$, $\mu_2$ is the mean function between the change point  $\floor{ns_{1}}$ and  second change point  $\floor{ns_{2}}$ and so on. We are interested in those change  points in $S$, where the maximal absolute deviation between the mean functions before and after the change point exceeds a given threshold, say $\Delta >0$. The choice of this threshold  is  problem specific 
as  it defines, when a change is considered as relevant. For the human movement data used in this work we discuss this choice in more detail in Section \ref{sec4}.

 Our goal is to detect all locations of relevant changes, that is to estimate the set
\begin{align} \nonumber % \label{d3} 
    S_{\rm rel}&=\big \{s_i \in \{ s_1,...,s_m \}  | ~\| {\mu_{i+1}-\mu_{i}} \|_\infty >\Delta \big  \}   
\end{align} 
and  control  the statistical error of the estimator. 
 For this  purpose  we proceed in two steps: 
\medskip
 \textbf{{Step 1:}} We estimate (consistently)     the number of all change  points and their locations by a particular binary segmentation algorithm. Note that in this step, we recover change points  which may or may not be of a relevant size. We denote the estimated number and the estimated set of 
    change points  by  $\hat m$ and  
    \begin{align}
        \hat S = \{\hat{s}_1, \hat{s}_2, \dots, \hat{s}_{\hat m}\}, 
    \nonumber %    \label{cp_set}
    \end{align} 
    respectively (we define $\hat{s}_0=0$ and $\hat{s}_{\hat m +1}=1$).
    This part of the procedure is described in  Algorithm \ref{basuAlg002} below.
 \medskip 
 \newline
  \textbf{{Step 2:}} 
     We investigate for each interval $[\hat s_{i-1}, \hat s_{i+1}]$  ($i=1, \ldots , \hat m$) if it contains  a relevant change point. This is done by  defining for each interval detectors, say   $\hat T_{n,1}, \ldots , \hat T_{n ,\hat m}$, which are then aggregated by the maximum operator, that is $\hat T_n = \max_{i=1, \ldots , \hat m} \hat T_{n,i}$.
     %,  where $\hat m$ denotes the estimate of the number of all change points from Step 1.
     Next we calculate a quantile, say $q_{1-\alpha}^*$ by a non-standard  bootstrap procedure such that 
   $$
    \limsup_{n \to \infty}\p(\hat T_n>q^*_{1-\alpha})\leq \alpha~
    $$
    whenever $S_{\rm rel} = \emptyset $.
   %for testing the hypotheses \eqref{mult_null_hyp}.   
   Finally, we   define the set of estimated relevant change points by
 \begin{align}
     \label{d2}
     \hat S_{\rm rel}&= \big \{\hat s_i \in   \{ \hat s_1,...,\hat s_{\hat m} \}  \big | ~ \hat T_{n,i}>q^*_{1-\alpha} \big  \}~.
 \end{align}

Note that in the case where Step 1 does not recover any change points, Step 2 is not  executed.  The details of  Step 1 and 2 are explained 
in the following Section \ref{sec21}.
 and \ref{sec22}, respectively.

\subsection{A first segmentation of the functional time series}
\label{sec21}

Binary segmentation is a very well known algorithm  for the multiple change detection scenario.  This algorithm  scans through the whole data set and looks for a first change point. When such a change point is detected, the data is divided into two samples before and after this initial change point and the procedure is repeated for each of the segments. The algorithm stops as soon as all the segments so far return a lack of change point. Since its introduction by  
\cite{voi1981}
 binary segmentation has been used and investigated by numerous authors \citep[see][among  many others]{fryzlewicz2014wild,baranowskietal2019,kovacs2023seeded}. However the theoretical properties of the algorithm for functional data are less well understood.
In order to identify all change points $s_1 , \ldots , s_m $ and the number of change points $m$ in the present context we define  for $l, r \in \{ 1, \ldots , n\} $  $( l < r) $, $ t \in [0,1]$, 
the sequential processes (in  $s \in [ \frac{l+1}{n},\frac{r}{n} ]$) 
 \begin{align}  
 \label{hd1}
        \hat U_{l,r}(s, t) = \frac{1}{r-l} \Big ( \sum_{j =l+1}^{\lfloor sn \rfloor} X_{j} (t) + (r-l) \Big  (s- \frac{\lfloor s(r-l) \rfloor }{(r-l)} \Big ) X_{ r } (t) - \frac{\lfloor sn \rfloor-l}{r-l}\sum_{j= l+1}^{ r} X_{j} (t)\Big ),   
    \end{align}
    and determine the change points by Algorithm \ref{basuAlg002}. We will establish  in Theorem  \ref{consistent_num} in Section \ref{sec4}  the validity  of this procedure. In other words: Algorithm \ref{basuAlg002} provides consistent 
 estimators for both the location and the number of the change points.
    
\begin{algorithm}
    \begin{algorithmic}[1]
    \State \textbf{Fix} $t_0 \leftarrow $ first time point  and $T \leftarrow$ 
   last time point   of sample
    \State  \textbf{Fix} $ \hat  m \leftarrow 0$ (estimate of the number of change points)
        \State   \textbf{Fix} $ \hat S \leftarrow  \emptyset $ (set of estimated change points)
    \If{ $T- t_0 \leq 1$} \State STOP
    \Else   $\quad \hat{k}= \argmax_{s \in t_0 ,...,T } \norm{{U}_{t,T} (s/n, \cdot)}_2$   \\
    $\quad\quad  ~~~\mathcal{U}= \norm{{U}_{t_0,T} (\hat{k}/n, \cdot)}_2  $
    \If{ $\mathcal{U} > \xi_n$} 
     \State $\hat m \leftarrow \hat m+1 $ 
    \State  $\hat S  \leftarrow  \hat S  \cup \{ \hat k\} $ 
    \State  \textbf{run} BINSEG ($t_0, \hat{k},  \xi_n$) and BINSEG ($\hat{k} , T, \xi_n$)
    \EndIf
    \EndIf
    \State \textbf{Order} the detected change points,  i.e. $ \hat S = \{ \hat{k}_{1}, \hat{k}_{2}, \dots, \hat{k}_{\hat m}\} $  
    with $\hat{k}_1< \hat{k}_2  < \dots <  \hat{k}_{\hat m}$
   
     \State {\textbf{For} $i=1, \ldots , \hat m $ define:  $\hat s_0   \leftarrow 0$; 
      $\hat s_i   \leftarrow \hat{k}_{
     i}/n $;  $\hat s_{\hat m +1}    \leftarrow 1$}  
    \end{algorithmic}
    \caption{\textbf{Function:} BINSEG($t_0, T,  \xi_n$), for estimating change points in the sequence $X_{t_0}, \ldots , X_T$.}
    \label{basuAlg002}
\end{algorithm}

In the upper part of Figure \ref{fig:basuFSig01} we illustrate the application of Algorithm \ref{basuAlg002} for  
a full functional time series of biomechanical knee angle data, collected  from one  athlete. The individual component curves which correspond to each cycle during the course of the run (seen clearly in the right panel of  Figure \ref{fig:runner_n_angles}) are not visible here as the data is collected via a fatiguing protocol and hence consists of a large number of cycles.  The red vertical lines show the estimated change points found using Algorithm \ref{basuAlg002}. Note that as seen, not all change points determined in this step might be relevant. In  the  lower   part of Figure \ref{fig:basuFSig01}
 we display the differences of the estimates of the mean functions 
on the four segments identified by the three estimated change points. 
We observe that the differences 
between the two segments separated by  the first change point   are 
 much larger than the other two. As will be seen later, some changes detected in this step could be due to unforeseen obstructions in data collection such as shifting of sensors due to sweating, pedestrians on the path etc.

\begin{figure}
\centering
\hspace*{-40pt}
\includegraphics[scale = 0.48]{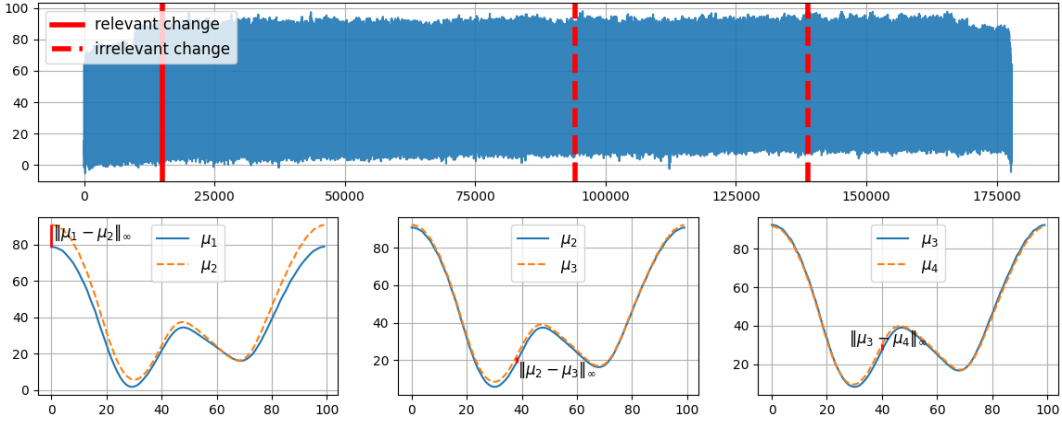}
\caption{\it Biomechanical knee angle data for a single runner. Upper part: The full functional time series with vertical lines denoting change points detected by Algorithm \ref{basuAlg002}. The crosses denote the relevant change points detected by Algorithm \ref{basuAlg001}: only the first change point is relevant.
   Lower part: Mean functions  on the different segments determined by the estimated change points, with the red vertical lines showing the maximum deviation between them. 
  The sample size for this runner is $n = 1800$.   }
    \label{fig:basuFSig01}
\end{figure}

\subsection{Hunting for relevant change points}
\label{sec22}
In this section we describe the details for Step 2 of our approach. In particular we will 
 motivate a new bootstrap procedure, which is used to define quantiles, such that 
 all relevant change points  in the functional time series   are detected with a statistical guarantee.
  More specifically, recall that  
  $\hat{k}_{1} < \hat{k}_{2}<  \dots < \hat{k}_{\hat m}$ denote the estimated change  points from  Algorithm \ref{basuAlg002} and that   
  $0= \hat{s}_0  < \hat  s_1 <  \ldots   < \hat s_{\hat m} <  \hat s_{\hat m +1}=1 $ denote the corresponding values scaled to the interval $[0,1]$. Further, for $i=1, \ldots ,  \hat m$  let $\hat h_i$ 
denote the linear transformation from the interval $[\hat s_{i-1},\hat s_{i+1}]$  onto $[0,1]$, that is 
\begin{align}
\nonumber %    \label{d3}
\hat h_i (s) = {  \frac{s-\hat s_{i-1}}{ \hat s_{i+1}-\hat s_{i-1}} }~.  
\end{align} 
and let 
\begin{align}
\nonumber %    \label{d3a}
 h_i (s) = { \frac{s-  s_{i-1}}{ s_{i+1}- s_{i-1}}  }
\end{align}
be the deterministic analog $(i=1, \ldots m$). We recall the definition  of the CUSUM statistic \eqref{hd1} for a given functional data stream and note that it can be shown that
\begin{align}
   \E[ \hat U_{\hat k_{i-1},\hat k_{i+1}}(s,t)]\simeq (h_i(s \land s_i)-h(s) h_i( s_i))(\mu_{i}(t)-\mu_{i+1}(t))+o_\p(1), \nonumber % \label{expectation_Uhat_process}
\end{align}
% uniformly with respect to $s$ and $t$.
Therefore,  the statistic 
\begin{align} \nonumber %
    \hat M_{n,i}= 
    \norm{ \hat U_{\hat k_{i-1},\hat k_{i+1}}}_\infty 
\end{align}
is a reasonable estimate of the size of the change $h_i( s_i)(1-h_i(s_i))\|{\mu_{i+1}-\mu_{i}}\|_\infty$ at the point $s_i$. To check if the change in the $i$th interval is relevant we will compare $ \hat M_{n,i} $  with $ h_i(s_i)(1- h_i(s_i)) \Delta$ using the detector
\begin{align}
\label{d4}
    \hat T_{n,i}=\sqrt{\hat n_i}\big ( \hat M_{n,i}-\hat h_i(\hat s_i)(1-\hat h_i(\hat s_i))\Delta\big )
\end{align}
where $\hat n_i = \lfloor n \hat s_{i+1} \rfloor-   \lfloor n \hat s_{i-1} \rfloor$ (the scaling factor $\sqrt{\hat n_i} $ becomes clear later, see Remark \ref{rem1a} below). The different statistics are then aggregated by the maximum operator, that is 
\begin{align}
    \hat T_n:=  \max_{1 \leq i \leq \hat m} \hat T_{n,i} ~.
    \label{basuT00}
\end{align}
We would like to  use quantiles of this statistic to detect relevant change points. 
However, it turns out that even asymptotically the   distribution of $\hat T_n$ depends in a complicated manner on several parameters, which are difficult to estimate
(see Remark \ref{rem1a} below).
As an alternative we will develop a non-standard bootstrap method to obtain (asymptotically)  upper bounds for these quantiles.
To explain our approach we  denote the sample means 
before and after the estimated change point $\hat k_i$, calculated
from the sample $\{ X_{k_{i-1}+1}, X_{k_{i-1}+2} \ldots , X_{k_{i+1}}\}  $, by
\begin{align*}
    \hat{\mu}_{1,\hat k_i} = \frac{1}{\hat k_i      -  \hat k_{i-1} } \sum_{j = \hat k_{i-1}+1}^{\hat k_i } X_{ j} \quad \text{ and } \quad \hat{\mu}_{2,\hat{k}_i} = \frac{1}{ \hat k_{i+1}  -  \hat k_i} \sum_{j = \hat k_i  + 1}^{\hat k_{i+1}} X_{j} ~.
\end{align*}
Further, let $(\xi_k^{(1)}: k \in \N ), \dots, (\xi_k^{(R)}: k \in \N ) $ denote independent sequences of i.i.d. standard normal distributed random variables. 
 On the rectangle
$[ \hat s_{i-1} , \hat s_{i+1}] \times [0,1]$ we consider the processes $\hat{B}_{i,n}^{(1)}, \dots, \hat{B}_{i,n}^{(R)}$ defined 
by  
 \begin{align}
\nonumber % \label{p4}
     \hat{B}_{i,n}^{(r)} (s, t) &= \frac{1}{\sqrt{\hat{n}_i}} \sum_{k= \hat k_{i-1}}^{\floor{sn}} \frac{1}{\sqrt{L}} \bigg( \sum_{j=k}^{k+L-1} \hat{Y}_{j} - \frac{L}{\hat n_i} \sum_{j= \hat k_{i-1}}^{\hat k_{i+1}} \hat{Y}_{j} (t)\bigg) \xi_k^{(r)}\\
     & + \sqrt{\hat n_i} \bigg( s- \frac{\floor{s \hat n_i}}{\hat n_i}\bigg) \frac{1}{\sqrt{L}} \bigg( \sum_{j= \floor{s n} +1}^{\floor{s n}+L} \hat{Y}_{j} - \frac{L}{n} \sum_{j= \hat k_{i-1}}^{\hat k_{i+1}} \hat{Y}_{j} (t)\bigg)\xi_{\floor{sn}+1}^{(r)},
     \nonumber %
 \end{align}
 where 
 \begin{align}
     \hat{Y}_{j}  = X_{j} - (\hat{\mu}_{2, \hat{k}_i}- \hat{\mu}_{1, \hat{k}_i}) \indicator{ j>\floor{\hat{s}_i n}} , \quad j = 1, \dots, n \, ;\, 
     \nonumber %
 \end{align}
 and $L \in \N$ defines the  length  of the  blocks and satisfies $L /n \rightarrow 0$  and $L \rightarrow \infty$ as $  n \rightarrow \infty$.  
We then consider a centered 
 bootstrap analogue 
 \begin{align}
\nonumber % \label{p5}
     \widehat{W}_{i,n}^{(r)} (s, t) = \hat{B}_{i,n}^{(r)} (s, t) - \hat h_i(s) \hat{B}_{i,n}^{(r)} (\hat s_{i+1}, t), 
 \end{align}
of the process $\hat U_{\hat k_{i-1},\hat k_{i+1}}$  and define for each  estimated change point $\hat{s}_i$   the set
       \begin{align}
        \hat{\mathcal{E}}_{i}^{\pm} = \Big  \{ t \in [0,1] : \pm (\hat{\mu}_{1, \hat{k}_i} (t)- \hat{\mu}_{2, \hat{k}_i} (t)) \geq \| \hat{\mu}_{1, \hat{k}_i} - \hat{\mu}_{2, \hat{k}_i} \|_\infty  - \frac{c \log  n}{\sqrt{n}} \Big \}, \label{extremal_sets}
    \end{align}
where  $ c>0$ denotes a constant. It can be shown that these sets are  consistent estimates of  the extremal  sets 
\begin{align}
\nonumber %\label{d7}
    \mathcal{E}_i^{\pm}:= \big \{ t\in [0,1] \big | ~  \mu_{i}(t)-\mu_{i+1}(t) =  \pm \| \mu_{i}-\mu_{i+1} \|_\infty \big  \},    
\end{align}
which contain the points, where the function  $t \to  ({\mu_{i+1} (t) -\mu_{i} (t) }) $  attains its sup-norm 
$\| \mu_{i+1}-\mu_{i}\|_\infty  $. 
Note that these sets depend on the (unknown) mean functions before and after the $i$th change point.
We  then suggest to consider
the bootstrap statistics
\begin{align} \label{d21}
    \hat T_{i}^{(r)} &= \max \Big\{ \sup_{t \in \mathcal{\hat E}_{i}^+} \widehat{W}_{i,n}^{(r)} (\hat s_i, t) , \sup_{t \in \mathcal{ \hat E}_{i}^-} (-\widehat{W}_{i,n}^{(r)}(\hat s_i, t))\Big \} 
\\
\label{d22}
\hat T_n^{*,(r)} &= \max_{i=1}^{\hat m}  \hat T_{i}^{(r)} 
\end{align}
and define  
$      q^*_{1-\alpha}$ 
 as the $(1-\alpha)$-quantile of bootstrap statistic $\hat T_n^{*,(r)} $. 
Finally,  the set of relevant change points  is estimated by \eqref{d2}.
 The implementation of this procedure is described in Algorithm \ref{basuAlg001}, which defines an estimate of $      q^*_{1-\alpha}$ by resampling.

\begin{algorithm}[t]
	\begin{algorithmic}[1]
	\State \textbf{Compute} the number $\hat m$ and all (ordered) change points $\hat{S} = \{\hat{s}_1, \dots, \hat{s}_{\hat m} \}$ using BINSEG($1,n, \xi_n$)
		\State \textbf{Fix} { block length $L$,  number of bootstrap replications $R$ and constant $c>0$ }
  \For{ $1= 1, \ldots , \hat m$ }
		\State \textbf{Compute}  the estimates $\hat {\mathcal{E}}_i^{\pm} $ of the extremal sets 
defined 
 in \eqref{extremal_sets}
		\EndFor
         \For{ $r = 1, \dots, R$ %  and $s, t \in [0,1]$
         }
         \State \textbf{Compute} $\widehat{W}_{i,n}^{(r)} (s, t) $
         \If{ $\hat {\mathcal{E}}_i^{\pm} \neq \emptyset $}
         \State $\hat T_{i,n}^{(r)} =  \max \{  \sup_{t \in \hat {\mathcal{E}}_i^+} \widehat{W}_{i,n}^{(r)} (\hat s_i, t) ,  \sup_{t \in \hat {\mathcal{E}}_i^-}  -\widehat{W}_{i,n}^{(r)} (\hat s_i, t)\}$
         \Else   \textbf{  set} $\hat T_{i,n}^{(r)} = \pm \infty $
         \EndIf
         \State \textbf{Compute} $\hat T_n^{*,(r)} = \max_i \hat T_{i,n}^{(r)}$
         \EndFor	
	\State \textbf{Compute  }  $q^*_{1-\alpha} \leftarrow $ 
 as the   empirical $(1-\alpha)$-quantile of bootstrap sample  $\hat T_n^{*,(1)} . \ldots , \hat T_n^{*,(R)} $ 
	\State \textbf{Estimate } the relevant change points by \eqref{d2}
	\end{algorithmic}
 \caption{Estimation of relevant  change points}
 \label{basuAlg001}
\end{algorithm}

\begin{remark} \label{rem1a}
    {\rm  We give some motivation 
    why this bootstrap procedure can be used to generate quantiles in the present context. 
    If all changes are not  relevant, that is $\max_{1 \leq k \leq m}\norm{\mu_{k}-\mu_{k-1}}_\infty \leq \Delta$, it can be shown 
    (see Section \ref{sec6} for details)
    that 
    the statistic $\hat T_n$ in \eqref{basuT00} is bounded by
    $ \hat D_n:=\max_{1 \leq i \leq \hat m}  \hat D_{n,i}$ with probability converging to $1$  as $n \to \infty$, where
\begin{align} \label{d4a}
 \hat D_{n,i}:=\sqrt{\hat n_i}\left(\hat M_{n,i}-\hat h_i(\hat s_i)(1-\hat h_i(\hat s_i))\norm{\mu_{i+1}-\mu_{i}}_\infty \right). 
\end{align}
Moreover, there even is equality if $\norm{\mu_{i+1}-\mu_{i}}_\infty = \Delta$ for all $i=1, \ldots , m $.
One can show (see Section \ref{sec6} for details) that 
    \begin{equation}
    \label{d9}
    \begin{split} 
    \hat D_{n,i} & \overset{\mathcal{D}}{\rightarrow} D(\mathcal{E}_i)  :=\max \Big \{\sup_{t \in \mathcal{E}_i^+}\mathbb{W}(h_i(s_i),t), \sup_{t \in \mathcal{E}_i^-}\mathbb{W}(h_i(s_i),t)  \Big\}, \\
    \hat D_n  :=\max_{i=1}^m    \hat D_{n,i} & \overset{\mathcal{D}}{\rightarrow}  D(\mathcal{E}):=\max_{i=1}^m  D(\mathcal{E}_i) \\
    & ~~~~~~~~~~~~ =
    \max_{i=1}^m \max \Big \{\sup_{t \in \mathcal{E}_i^+}\mathbb{W}(h_i(s_i),t), \sup_{t \in \mathcal{E}_i^-}\mathbb{W}(h_i(s_i),t)  \Big\}
    ~.
        \end{split} 
    \end{equation}  
Here the symbol $\overset{\mathcal{D}}{\rightarrow}$ denotes weak convergence of real valued random variables and $\mathbb{W}$ is a mean zero Gaussian process 
on $[0,1] \times [0,1] $ with covariance structure given by \begin{align} 
\nonumber % \label{d6}
      \text{\rm Cov}(\mathbb{W}(s,t),\mathbb{W}(s',t'))=(s\land s' -ss')  \sum_{i=-\infty}^{\infty} \text{\rm Cov}(\epsilon_{0}(t),\epsilon_{i}(t))~.
\end{align}
The bootstrap statistics 
$\hat T_i^{(r)}$ and $\hat T_n^{*,(r)}$
in   \ref{d21}, \ref{d22} mimic asymptotically the distributional properties 
of the random variables $\hat D_{n,i}$ and $\hat D_n$
defined in \eqref{d4a} and  \eqref{d9}, respectively. As $\hat T_n$ is asymptotically bounded by $\hat D_n$, we have 
$$
\mathbb{P} \big ( \hat T_n \leq  q_{1-\alpha}^* \big ) 
\geq \mathbb{P} \big ( \hat  D_n > q_{1-\alpha}^* \big ) + {o}(1)=  1 - \alpha +{o}(1) ~
$$
as $n \to \infty$,  and this statement is sufficient for proving consistency of the estimated set of relevant change points \eqref{d2} for $S_{\rm rel}$.
Details  can be found in Section \ref{sec4}.
}
\end{remark}

\begin{remark}  ~~\\ 
{\rm 
 
\begin{itemize}
    \item[(a)] It is  possible develop a similar  algorithm for the $L^2$- instead of the sup-norm. In fact, such a procedure  would not require the estimation of the extremal sets.  However, the choice of the threshold for this norm is conceptually 
(not mathematically) more difficult, and we   emphasize that  one of the main motivations for using  the sup-norm in our approach is the easy  and very  natural interpretation of the threshold in the considered application.     
       \item[(b)]
       In this paper we assume that at each time point  the full trajectory is observed.
        Our method is also applicable  (and its validity can be established) to dense and discrete observations from the trajectory. In particular, in the considered application the 
        hip, knee and ankle angles are continuously recorded resulting in an extremely dense grid, where the trajectories are observed.
 \item[(c)] 
Note that alternatively to the above retrospective change detection,  sequential change point analysis of such a data is of prominent interest in the goal of injury prevention. One  aim in this collaboration project \textit{Sports, Data, and Interaction\footnote{\url{http://www.sports-data-interaction.com/} }}  is to set up an interactive system to provide real-time feedback to the athlete.  In this scenario, the  idea is to have the least invasive approach of recording biomechanical joint angles, for example, via a video camera. When such a data is studied by sequential analysis, the athlete or trainer could be alerted of a change  and that the joint under consideration is moving differently and appropriate updates to training strategies may to be made. This is however a future direction that may be pursued and we do not look at this scenario here. 
    \item[(d)]  The data set under consideration consists of $3$-dimensional functional time series (for the hip, knee and ankle angles).  In the discussion so far we have developed methodology for finding multiple change points in  each component separately. 
Our approach can be generalized to find multiple change points in  multivariate functional time series, and we briefly indicate here how this can be done. 
\\
For the first step we propose to use binary segmentation for each of the coordinates (of course, any other consistent procedure can be used).
    For the second step we have multiple ways to define what a relevant change is, depending on whether or not we would like to consider changes in the different coordinates jointly or separately. For the latter one could simply search for relevant changes (in the sense we have used previously) in any of the coordinates, i.e. for each  coordinate $\ell =1,2,3$ we define  the analog of the statistic \eqref{d4}, say
    $$
       \hat T_{n,i}^{(\ell )} =\sqrt{\hat n_i}\big ( \hat M_{n,i}^{(\ell )}-\hat h_i(\hat s_i)(1-\hat h_i(\hat s_i))\Delta_\ell\big )
    $$
    with thresholds $\Delta_1,\Delta_2,\Delta_3$,  respectively, and then additionally take the maximum over all three coordinates in \eqref{basuT00}, that is
   \begin{align}
       \label{basuT00a}
     \hat T_n:=  \max_{1 \leq i \leq \hat m} \max_{\ell=1}^3 \hat T_{n,i}^{(\ell )} ~.
   \end{align}
    If we want to consider the changes in all coordinates simultaneously we could  first  aggregate the statistics $\hat M_{n,i}^{(\ell )}$ using a functional such as $\phi (x) = \| x \|_q = \big ( \sum_{\ell=1}^3 |x_i|^q \big )^{1/q}$ ($1 \leq q \leq \infty$) and then rescale it, that is 
    \begin{align*}
        \hat T_{n,i} =\sqrt{\hat n_i}\big \{  \phi \big ( (  \hat M_{n,i}^{(1 )}, \hat M_{n,i}^{(2 )} , \hat M_{n,i}^{(3 )} )\big ) -\hat h_i(\hat s_i)(1-\hat h_i(\hat s_i))\Delta\big \}  ~.
    \end{align*}    
After this we can  proceed as in \eqref{basuT00}.
      Multivariate analogs of Algorithm \ref{basuAlg001} can  be developed for both cases and  their  validity can be shown by similar arguments as given in the proof Theorem \ref{thm3}.  As the foremost goal of the {\it Sports, Data and Interaction} team was injury prevention in the knee joint we do not pursue these approaches here  further.
\end{itemize}
}
\end{remark}

\section{Finite sample properties}
\label{sec3} 
  \def\theequation{4.\arabic{equation}}	
\setcounter{equation}{0}

In this section we study the finite sample properties of Algorithm   \ref{basuAlg001} by means of a simulation study and illustrate its practical application in two data examples consisting of fatigue data from a laboratory setup as well as from a marathon run.

Note that  Algorithm \ref{basuAlg002} and \ref{basuAlg001} require the specification of several tuning parameters. This includes firstly,  the threshold value for locating change points in the binary segmentation $\xi_n$ in Algorithm \ref{basuAlg002}.  
Following a recommendation in 
\cite{rice2022consistency} 
we use $\xi_n = \hat{\sigma}_n \sqrt{3 \log(n)}$, where 
$\hat{\sigma}_n^2 = \text{median}\left(\|X_{i+1} - X_i\|^2_2 /2, \, i = 2, \ldots, n\right)$.  For the parameter $c$ required  in the estimation of the extremal sets in \eqref{extremal_sets} we choose  $c = 0.1$ as recommended 
by \cite{dette2020functional}. Furthermore, an initial robustness analysis showed that the procedure is rather stable with respect to this choice of $c$ (the details are omitted for the sake of brevity). The number of bootstrap repetitions  is always fixed as $R = 1000$. The block-length parameter $L$ in the multiplier bootstrap  is selected by a  method proposed by  \cite{rice2017plug}, which requires the choice of a weight function. We have implemented their algorithm for the  Bartlett, Parzen, Tukey-Hanning and Quadratic spectral kernel and  observed that the quadratic spectral weight function yields  the most stable 
 results  (these 
 findings are not displayed for the sake of brevity). 
 Finally, unless stated otherwise we use $\alpha = 10\%$ for  the  quantile $q^*_{1-\alpha} $.

\begin{figure}[t]
    \centering
%\hspace*{-0.075\textwidth}
    %\includegraphics[scale = 0.45]{images/simulation_2cp.pdf}
    %\includegraphics[scale = 0.45]{images/simulation_3cp.pdf}
    \includegraphics[scale= 0.4]{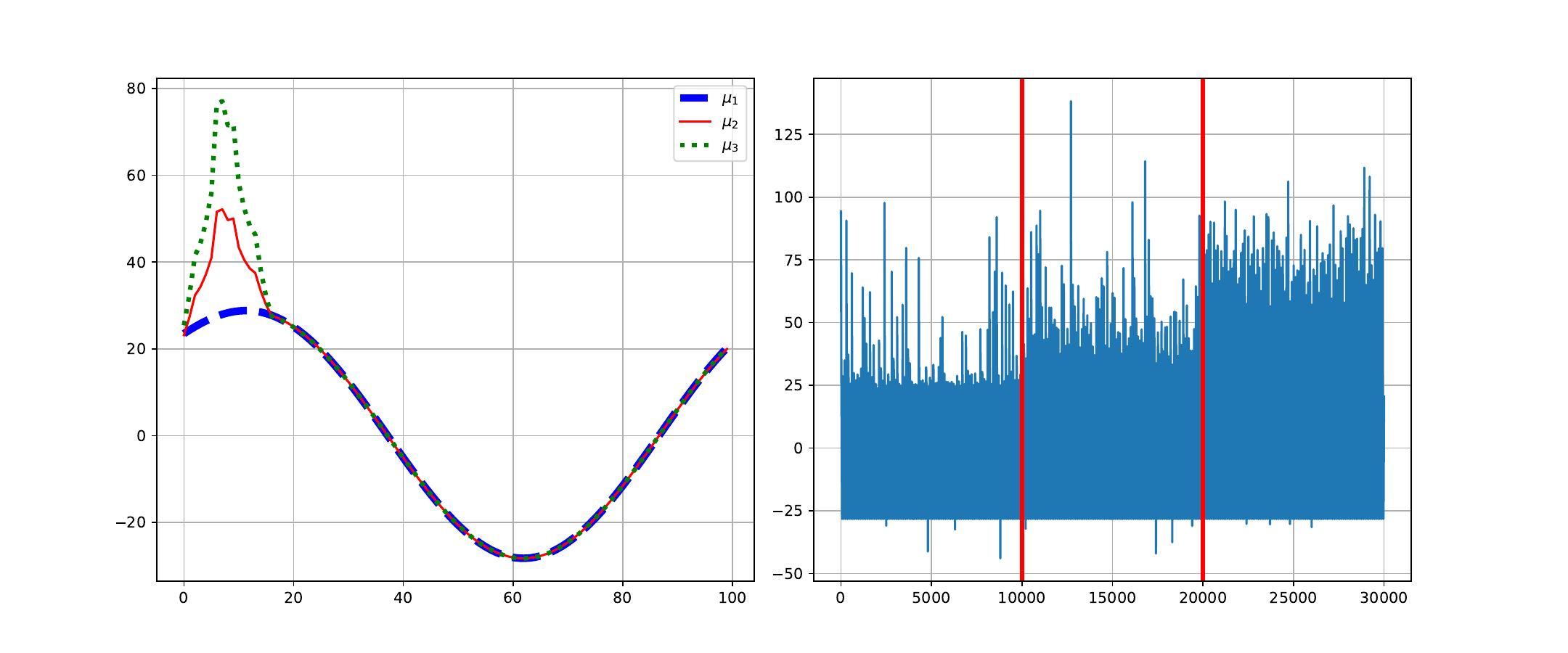}
    \includegraphics[scale= 0.4]{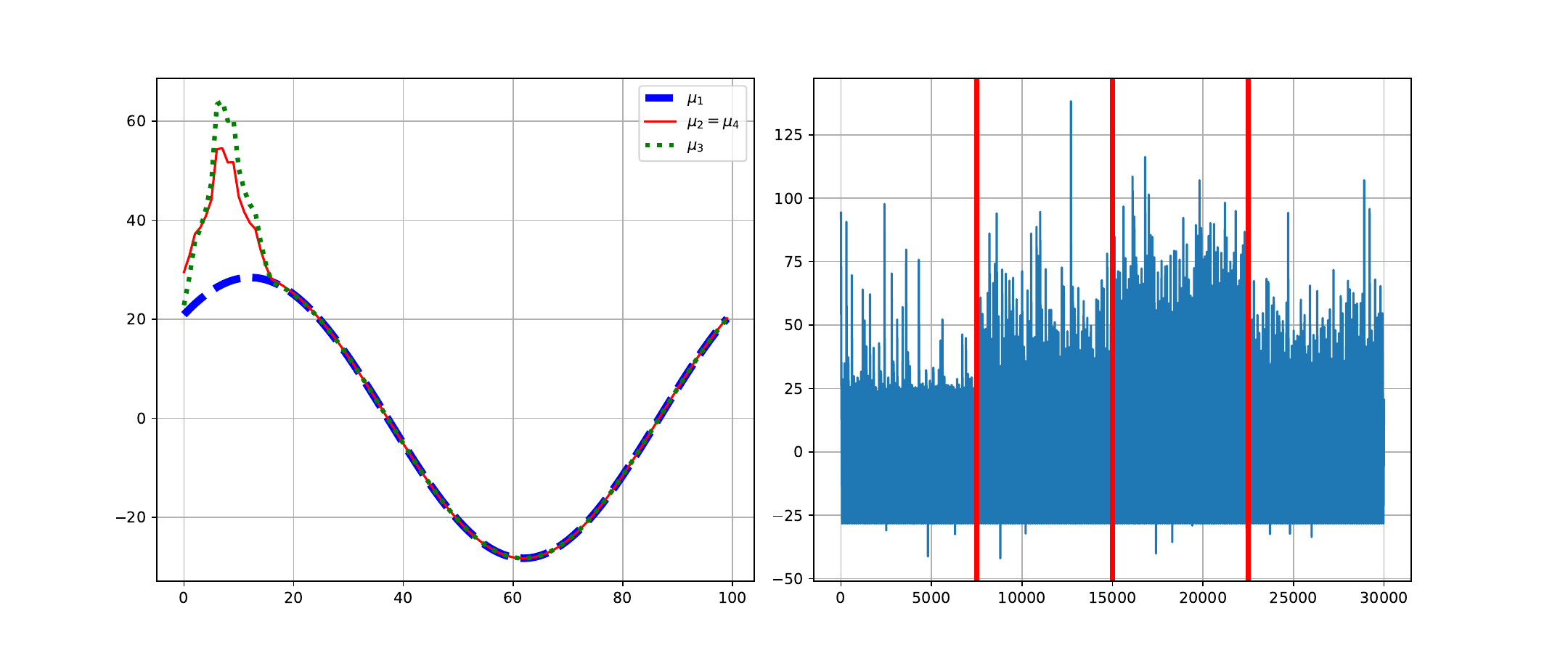}
    
    \caption{\it    
    Simulated data for  $n= 300$ observations with
    two changes (first row) and   three changes (second row).  Left panels:  mean curves on the different segments.   Right panels: Full time series data with change points given by vertical lines. 
 } 
    \label{fig:basuSimIntro}
\end{figure}

\subsection{Simulation study}
\label{sec31} 
For the simulation study we generate data  according to the  model in \eqref{eqn:model} 
with two and three change points. In the example with  two change locations, we choose 
$s_1= \floor{\frac{n}{3}} $,  $s_2= \floor{\frac{2n}{3}}$  
 and the mean functions are given by 
\begin{align}
    \mu_{j} (t)= \begin{cases}
        20 \big (\sin{ (2 \pi t) } + \cos (2 \pi t )  \big ) 
        %\quad \quad \quad \quad \quad \quad\quad \quad\quad \quad\quad \,\,     0< j \leq \lfloor  s_1 n \rfloor 
        \\
        20 \big (\sin{ (2 \pi t) } + \cos (2 \pi t )  \big ) + \Delta_J (t) \mathbbm{1}_{ [0, 0.16]}(t) 
        %\quad \quad \, \lfloor  s_1 n \rfloor < j \leq  \lfloor s_2 n \rfloor
        \\
        20 \big (\sin{ (2 \pi t) } + \cos (2 \pi t )  \big ) + 2\Delta_J (t)\mathbbm{1}_{[0, 0.16]} (t) 
        %\quad \, \, \,   \lfloor s_2 n \rfloor <j \leq n
    \end{cases}  ~~~~~~~j =1,2,3 , 
    \label{model_simulations}
\end{align}
where $\Delta_J \mathbbm{1}_{ [0, 0.16]}$ is a symmetric  cubic spline  on the interval $[0, 0.16]$  interpolating the points $(0.01,2), (0.02, 5), (0.03, 9), (0.04, 10), (0.05, 12), (0.06, 15), (0.07, 22) ,  (0.08,25) $ and their reflections with respect to the vertical line through the point $0.085$.  In the example with three change points we choose   $s_1= \floor{\frac{n}{4}}$, $s_2=\floor{\frac{2n}{4}}$  and $ s_3  =\floor{\frac{3n}{4}}  $ and the mean functions $\mu_1$,  $\mu_2$,  $\mu_3$ are given  by \eqref{model_simulations}, while  $\mu_4(t) =  20 \big (\sin{ (2 \pi t) } + \cos (2 \pi t )  \big ) + \Delta_J (t)\mathbbm{1}_{[0, 0.16]} (t)$.

In both examples   $(\epsilon_j)_{j \ge 1} $ is an fMA(1) process defined in the simulation section of  \citep{dette2020functional}. To be precise, for  the B-spline basis functions $\nu_1, \ldots, \nu_{21}$  we consider the linear space $H = \text{span}\{\nu_1, \ldots, \nu_{21}\}\subset C([0, 1] $ and define independent processes $\eta_1, \ldots, \eta_n \in H  $ by
\begin{equation}
\eta_j =  \sum_{i=1}^{21}
\mathbf{1}_{[-4,4]}(N_{i,j})\nu_i, \quad j = 1, \ldots, n, 
\nonumber % \label{basuSim001}
\end{equation}
where $N_{1,j},  \ldots, N_{21,j}$ are independent, normally distributed random variables with $\mathbb{E} [ N_{i,j}]=0 $ and $\text{Var}(N_{i,j})  = \frac{1}{i^2}$ $(i = 1, \ldots, 21; j = 1, \ldots, n)$.  The fMA(1) process is finally given by $\epsilon_i = \eta_i + \Theta \eta_{i-1}$, where the operator $\Theta : H \to H$ (acting on a finite-dimensional space) is defined by $0.8 \, \Psi /  \sigma_1 (\Psi ) $, the elements in the position $(i,j)$ of the matrix $\Psi$ 
are normally distributed random variables with mean zero and standard deviation $1/(ij)$ and 
$ \sigma_1 (\Psi ) $ denotes the spectral norm of $\Psi$.

The simulated data is displayed  in  Figure \ref{fig:basuSimIntro},  where we  also show the mean curves on the different segments. The results of the simulations (with sample sizes $n = 200, 300, 600$) are displayed in  Figure  \ref{fig:3cp_histo} (two changes in the upper part,  three changes in the lower part), which contains   histograms of the estimated relevant change points obtained  from $1000$ repetitions.  As is evident from the clustering of the histograms, we recover the true location of the change points fairly accurately and as predicted by the asymptotic theory in Section \ref{sec4}, performance improves with increasing sample size.  

\begin{figure}[t]
    \centering
    \hspace*{-7pt}
    \includegraphics[scale =0.2912]{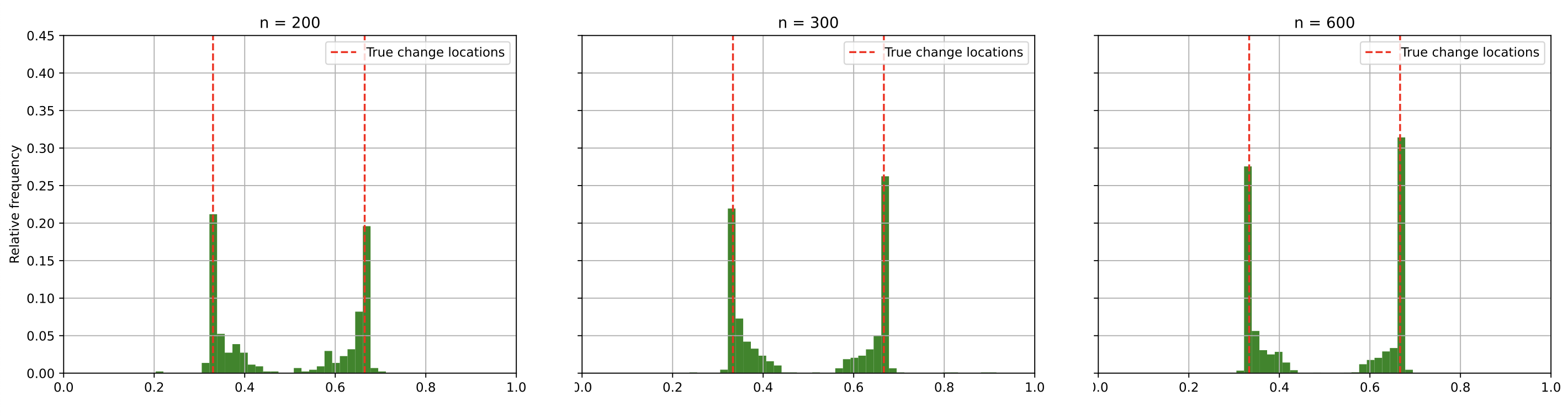}
    \hspace*{-7pt}
    \includegraphics[scale = 0.32]{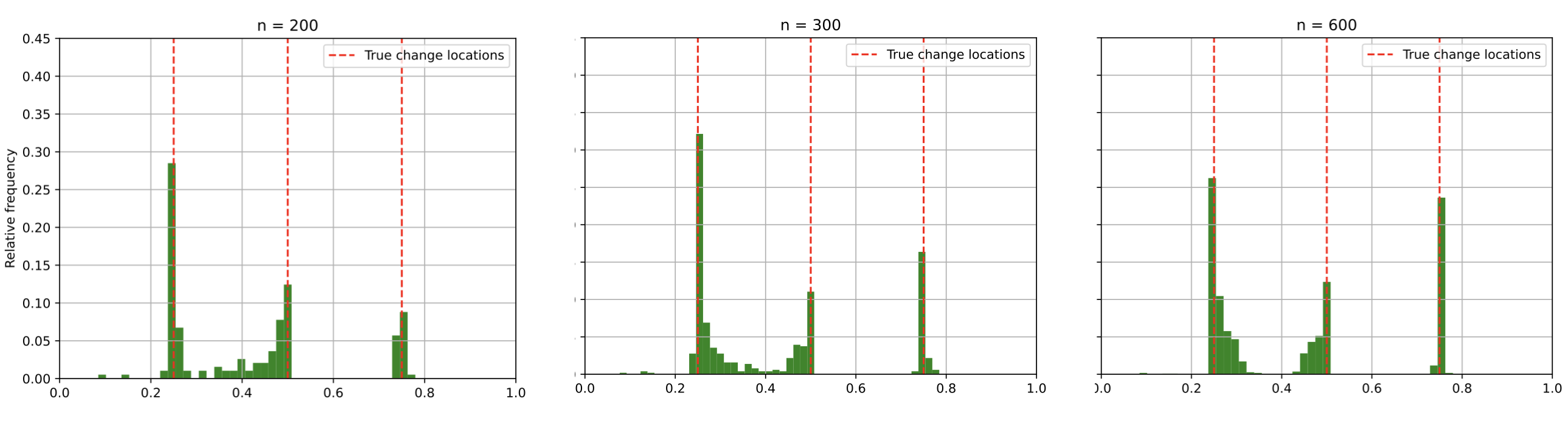}
    \caption{ \it Histograms of estimated relevant change points from $1000$ simulation with relative frequencies with respect to all changes recovered in Step 1.
    Upper part:  Two relevant changes. Lower part: Three relevant changes.}
    \label{fig:3cp_histo}
\end{figure}

\subsection{Data examples}
\label{sec5} 
  \def\theequation{5.\arabic{equation}}	
\setcounter{equation}{0}

In this section, we investigate two  data examples where the methodology developed  in this paper is directly applied for statistical  inference. 
One example refers to data  collected in the
laboratory, the other to data collected  during a  marathon run. Note that, hereafter the domain of each curve is the normalized time within each cycle of the data . A cycle in this regard means two subsequent contacts of the same foot with the ground. 

\subsubsection{Sensor data}
\label{sensor_data_section}
We investigate  biomechanical lower joint angle data from the knees from  a running athlete, which  was collected within the multidisciplinary project \textit{Sports, Data, and Interaction\footnote{\url{http://www.sports-data-interaction.com/} }} with the goal  of detecting persistent changes in the movement of the runner due to fatigue. Fatigue detection in running athletes is important for the elimination of problematic movements and reducing the risk of injuries. 
Similar analysis of the hip and ankle angle data is possible, even though this is not pursued further in this work.

The study subjects were recruited from a pool of frequent hobby runners between the ages of 21-60 years with experience ranging between 1- 14 years and  subjects reported running between 2-5 times per week. In this section, we  do not consider  professional runners. This is because we expect pronounced changes post-fatigue in the hobby runners while professional runners, being especially trained are able to endure higher fatigue levels and time constraints do not allow for recording such long datasets in the laboratory.

The data was collected by body-worn sensors via inertial measurement units (IMUs) produced by Xsens \citep[Xsens MVN link sensors, sampling at 240 Hz, see ][]{schepers2018xsens} while following a fatigue protocol, i.e. ensuring that an athlete was surely and steadily tired while minimising the impact of external factors on the data collection. The controlled laboratory setting minimizes the inevitable disturbances occurring during other data collection procedures like outdoors on a running track. Further, the treadmill controlled the speed of the runner and hence any deviations due to external factors other than fatigue were kept to a minimum. Unless otherwise stated, in  the following discussion we will look at the (right-) knee angle data from several runners on a treadmill.\\

For a proof of the validity of our  methodology we assume a decaying serial correlation structure   for the errors (see Assumption (A4) in Section \ref{sec4}). In order to verify that this is a reasonable assumption  for the data 
under consideration we estimate 
the lag $k$ autocorrelation 
$\gamma_k(t,s)=\text{Cor}(X_0(t),X_k(s))$   between  $X_0(t)$ and $X_k(s)$ by
$$
\hat \gamma_k(t,s) :=  \frac{\sum_{j=1}^{n-k} (X_j (t) - \bar X(t)) (X_{j+k} (s)- \bar X (s))}{\sqrt{\sum_{j=1}^{n-k} ( X_j (t) - \bar X(t))^2 \sum_{j= 1}^{n-k} (X_{j+k} (s) - \bar X (s))^2}}~.
$$
Exemplary, we display in Figure \ref{fig:variogram}   the 
aggregation by the $L^2$-norm 
$\|\hat \gamma\|_2$   of one runner for various lags. One can observe that the resulting function decays quite quickly but is substantially different from 0 for $k=1,2$ which validates our assumption. 
%Clearly the choice of bootstrap block bandwidth also depends sensitively on the correlation structure of the observations and we provide a data based approach for choosing it (see the beginning of Section \ref{sec3}), moreover we want to note that the results of the procedure are rather robust with respect to the choice of block width (see Table \ref{tab:different_block_lengths} and its discussion at the end of this subsection).
\begin{figure}[h]
    \centering
    \includegraphics[scale = 0.45]{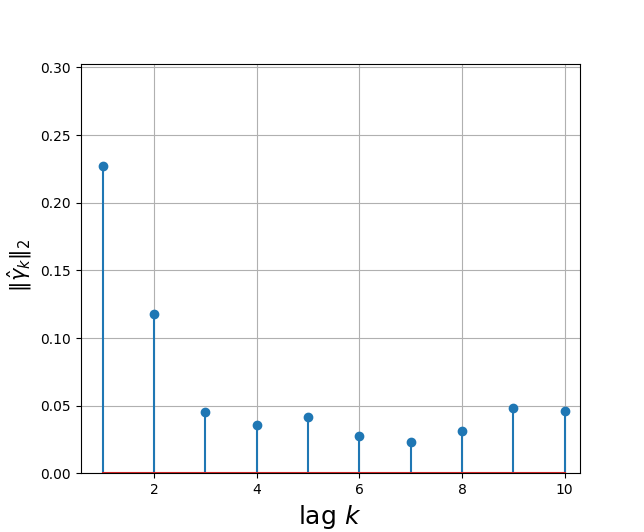}
    \caption{Variogram over lagged samples}
    \label{fig:variogram}
\end{figure}
%Similarly one can consider the sensitivity of our procedure with respect to the choice of threshold parameter $\xi_n$, just as for the block width we note that we have provided a data driven way to choose this parameter and that the results of the method are not too sensitive with respect to its choice (see Table \ref{tab:diff_xi_n} and its discussion). As long as the threshold is not chosen too large any additional change points detected due to a lower choice of $\xi_n$ in Step 1 will be discarded in Step 2.

In this type of application, the magnitude of adaptations of the biomechanical joints is an individual attribute of the respective runners (as seen in the examples in the following paragraphs). Consequently, the  
  choice of the threshold $\Delta$  has to be  individualised  for each runner. We address this problem  by estimating  mean curves from the initial $5\% $  and the  final $ 5\% $ of the
  curves and denote the corresponding estimates by   $\hat \mu_{\text{initial}}$ 
  and by  $ \hat \mu_{\text{final}}$. The threshold is then defined by   
 $\Delta = \norm{\hat \mu_{\text{initial}} - \hat \mu_{\text{final}} }_\infty /3 $
 %Further, we choose $\Delta =\frac{\Delta_{0}}{3}$ for the sensor data, 
 as we expect the runners to go from rest phase to pre-fatigued and finally fatigued phases.

In Figure \ref{fig:app_fig2}, we visualize the results of the methodology to detect relevant changes (denoted by crosses) due  to fatigue  for  a particular runner. Notice that the densely collected functional data has no obvious jumps and fatigue detection for such a data requires us to use sophisticated change point detection techniques. For this particular runner with sample size $n = 1800$ the relevant difference threshold is $\Delta = 6$  and the first two  change points  detected by Algorithm \ref{basuAlg002} are relevant, implying a significant change (i.e., an increase) in the knee angles at least up to $6$ degrees. Comparing the mean curves from each consecutive segment in the bottom part of Figure \ref{fig:app_fig2}, we obtain that the largest change is at the peak knee flexion which  means more bending of the knee during the swing phase (when the corresponding foot is in the air). This is in agreement with prior studies, as in \cite{zandbergen2023effects}, who  also have reported this in their preliminary analysis of fatigue data. 
 %Experts in the biomechanical field also report this increase in preliminary findings related to fatigued runners, see \cite{zandbergen2023effects}. 
 This increase in the knee flexion is attributed to decreased levels of stiffness in the runner during a fatigue protocol and is understood to be an attempt by runners to decrease the moment of inertia about the hip joint and could be attributed to a protection mechanism of the runner. Interestingly enough the second detected relevant change (at about $40\%$ of the stride cycle) occurs when the corresponding foot makes contact with the ground and may be attributed to be the cause of higher oxygen-cost in the later stages of the run. 

\begin{figure}[h]
\hspace*{-60pt}
    \centering
    \includegraphics[scale = 0.45]{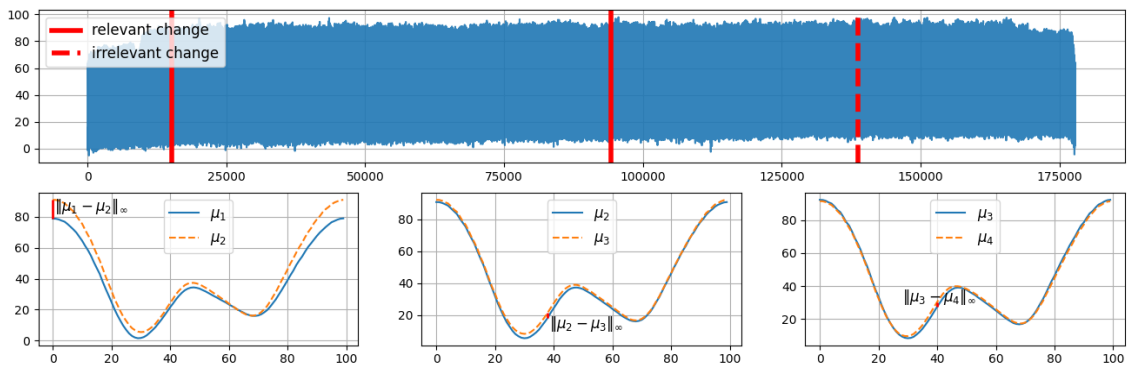}
    \caption{\it Biomechanical knee angle data for a single runner in indoor setting. Upper part: The full functional time series with crosses depicting that two detected changes are relevant. Lower part: Mean curves on the different detected segments and maximum deviations marked by (red) vertical lines. The sample size for this runner is $n= 1800$.} 
    \label{fig:app_fig2}
    %runner 2 in the dataset!
\end{figure}
 
 While the runner in Figure \ref{fig:app_fig2} shows an increased range of motion under fatigue conditions (an expected phenomenon) we display in Figure \ref{fig:alt_runner} the results for an alternative runner  , with sample size $n = 1800$. Here,  $\Delta = 5$  and Algorithm \ref{basuAlg001} identifies three relevant changes  along with a reduced range of motion.  At the first   change point, the difference is picked up at peak knee flexion (captured at $\norm{\mu_1- \mu_2}_ \infty$ in Figure \ref{fig:alt_runner}) during the stance phase corresponding to the phase of contact of the foot with the ground. This inevitably leads to the ground reaction forces causing an increased load on the knees. The second and third changes (of size  $\norm{\mu_2 - \mu_3}_\infty $ and $\norm{\mu_3- \mu_4}_\infty$ respectively) point towards a reduction in the knee extension and flexion respectively. This is an interesting, albeit concerning change  because it is contrary to observations for  most runners, where one expects reduced stiffness and hence increased range of motion, see \cite{apte2021biomechanical}. Such individual feedback on anomalies of the movement is valuable for future study of biomechanical data.

\begin{figure}[h]
\hspace*{-50pt}
\centering
\includegraphics[scale = 0.476]{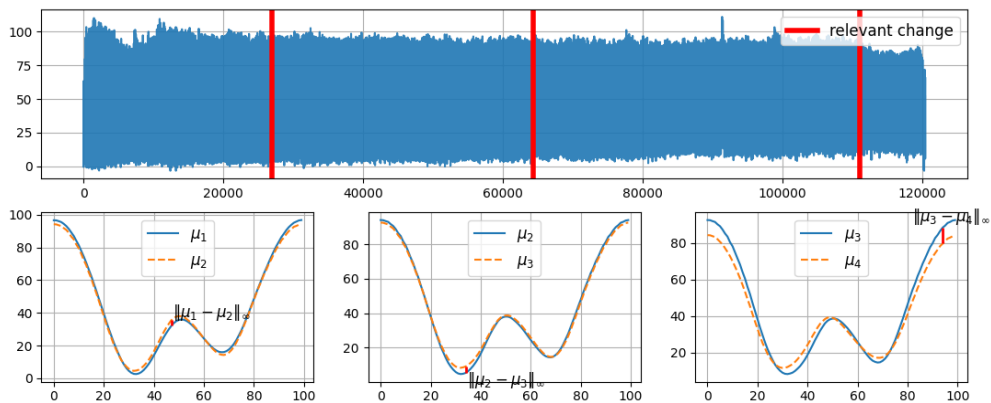}  
\caption{
\it 
Biomechanical knee angle data for an alternate runner in indoor setting showing interesting characteristics. Upper part: The full functional time series with crosses depicting that all detected changes are relevant. Lower part: Mean curves on the different detected segments and maximum deviations marked by (red) vertical lines. Sample size for this runner $n= 1205.$}
\label{fig:alt_runner}
% runner 3 in the dataset
\end{figure}

We conclude this section with a brief robustness study of Algorithm \ref{basuAlg001} with respect to
the choice of the  two tuning parameters, the block-length $L$ in  the bootstrap and the threshold $\xi_n$ in the estimators of 
the extremal sets. For this purpose 
we apply the  algorithm  with different choices of these parameters  to the Xsens data of a third runner. 
Here the  sample size is $n = 1400$ and is chosen as  $\Delta = 2.45$. 
% The results  are recorded in Table \ref{tab:different_block_lengths} and \ref{tab:diff_xi_n}. 
Table \ref{tab:different_block_lengths} 
shows the results for different 
choices of the block length. We 
observe 
that the same change points are estimated for a wide range of the  block length provided that this is not chosen too large.   
Similarly, we see from  Table \ref{tab:diff_xi_n}  that Algorithm \ref{basuAlg001} is rather robust with respect to the choice of $\xi_n$ provided that it is not chosen too large  
(in that case binary segmentation fails to pick up some of the relevant change points) or too 
 small (in that case the power of the procedure is diminished as the effective sample size of the intervals $[\hat s_{i-1},\hat s_{i+1}]$ is reduced too much). 
 Nevertheless, there is still a wide window of possible choices that produces the same change points and even in the more extreme cases the change points that are deemed relevant are consistent with the more moderate scenarios.
\label{sec64}
\begin{table}[h]
    \centering
    \begin{tabular}{|c|c|c|c|c||c||c|c|c|}
    \hline
         & $\alpha = 0.1$& $L= 3$& $L = 4$ & $L= 5$ & $L = 7$ & $L= 9$& $L=12$ & $L = 15$ \\
         \hline
         & first change &  526& 526 & 526 &  526& 526 &526 & 526\\
         & second change & 1112& 1112 & 1112& 1112 & 1112 & 1112 & -\\
         & third change &  - & 1137 & - & - & - & - &- \\
         \hline\hline
         & $\alpha = 0.05$& $L= 3$& $L = 4$ & $L= 5$ & $L = 7$ & $L= 9$& $L=12$ & $L = 15$ \\
         \hline
         & first change &  526& 526 & 526 &  526& 526 &- & -\\
         & second change & 1112& 1112 & 1112& 1112 & 1112 & - & - \\
         & third change &  - & - & - & - & - & - &- \\
         \hline
    \end{tabular}
    \caption{Relevant changes 
    in the  knee angle data of single runner
    detected by Algorithm \ref{basuAlg001} with 
    different block-lengths. }
    \label{tab:different_block_lengths}
\end{table}
\begin{table}[h]
    \centering
\begin{tabular}{|c|c|c|}
\hline
     $\alpha = 0.1$ & \textbf{Binseg changes} & \textbf{Relevant changes} \\
          \hline 
     $\xi_n = 4.25$ & $[260, 526, 1112, 1137, 1163, 1185]$ & 526, 1112, 1137\\
     $\xi_n = 5.5$ & $[526, 1112, 1137, 1163, 1185]$ & 526, 1112, 1137\\  
     $\xi_n = 6.1465$ & $[526, 1112, 1137, 1163, 1185]$ & 526, 1112, 1137\\
     $\xi_n = 7$ &  $[526]$ & 526\\
       \hline \hline 
     $\alpha = 0.05$ & \textbf{Binseg changes} &\textbf{ Relevant changes} \\
          \hline 
     $\xi_n = 4.25$ & $[260, 526, 1112, 1137, 1163, 1185]$ & 1112\\
     $\xi_n = 5.5$ & $[526, 1112, 1137, 1163, 1185]$ & 526, 1112, 1137\\
     $\xi_n = 6.1465$ & $[526, 1112, 1137, 1163, 1185]$ & 526, 1112, 1137\\
     $\xi_n = 7$ &  $[526]$ & 526\\
     \hline     
\end{tabular}
\caption{Robustness of multiple change point estimators with respect to the choice of $\xi_n$ by Algorithm \ref{basuAlg001}. Left part: 
estimated change points by  binary segmentation (step 1). Right part: estimated relevant change points.}
\label{tab:diff_xi_n}
\end{table}

\subsubsection{Marathon data}\label{rbAS002}

As an example for a rather uncontrolled environment we consider  marathon data,  
which gives rise to deviations due to factors external from the runner's fatigue. On the other hand, runners in such data sets are typically far more experienced and changes and adaptations in their biomechanical data are fewer and are of a smaller size. Details on the data sets used in this section, can be found  in \cite{zandbergen2023quantifying}. 
We choose $\Delta$ just as before in the sensor data example in \cref{sensor_data_section}.
An illustration of our methodology for one of the runners is shown in Figure \ref{fig:marathon_data_fig}, where $\Delta = 3$. This runner, with a fairly large sample size of $n = 22000$, shows an initial relevant adjustment to (perceived-) fatigue in the marathon 
%(given by $\norm{\mu_1 - \mu_2}_\infty$) 
followed by an adjustment that is below the relevant size.
%(at $\norm{\mu_2- \mu_3}_\infty$). 
Finally towards the end of the run, another relevant change to the knee movement occurs.
%(given by $\norm{\mu_3- \mu_4}_\infty$). 
We recover that both the relevant adjustments are in the flexion of the knee when the foot is in the air with an increase in the range of motion indicating reduced stiffness. Finally,  the irrelevant change still marks the tendency of the increase in flexion of the knee angle at this point (stance phase, foot is in contact with the ground)  during fatigue situation but is just not significant in this particular runner data and therefore not relevant for further analysis only in this case. Note that this is in contrast with some inexperienced runners where changes in the biomechanical data is also captured at this point of the stance phase; see Figure \ref{fig:alt_runner}, and \cite{maas2018novice} for a  further discussion on such increased kinematic changes in novice runners. This example  reiterates the usefulness of looking for relevant changes and not just at all changes in the data.

\begin{figure}
\hspace*{-55pt}
    \centering
\includegraphics[scale = 0.36]{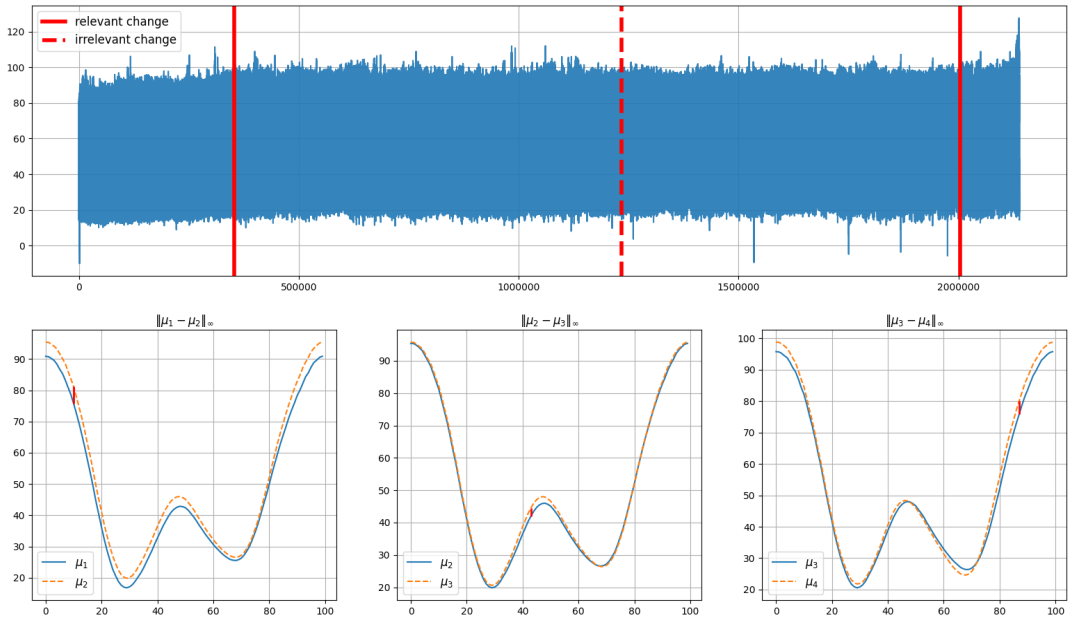}
\caption{ \it 
Biomechanical knee angle data of single marathon runner. Upper part: The full functional time series with two relevant changes (denoted by crosses) and one irrelevant change. Lower part: Mean curves on the different detected segments and maximum deviations marked by (red) vertical lines. The sample size for this runner  is $n= 22000$.
}
    \label{fig:marathon_data_fig}
    % runner 1 in the dataset
\end{figure}
\bigskip

\section{Theoretical analysis}
\label{sec4} 
  \def\theequation{4.\arabic{equation}}	
\setcounter{equation}{0}

In this section we establish the validity of our approach.  For  the theoretical analysis, we recall the definition of the triangular array in \eqref{eqn:model}, use the notations $X_{n, j}$,
$\epsilon_{n, j}$, $\mu_{n, j}$ etc.  to emphasize the dependence on $n$ 
and make the following  assumptions.
  
\begin{itemize}
    \item [{\bf(A1)}] 
    The error process   $\{ \epsilon_{n, j} |  \, j = 1, \dots, n \}_ {n \in \mathbb{N}}$ 
   is a  triangular-array of centered 
    $C ( [0,1])$-valued random variables
     (i.e. for fixed $n$ the process $(\epsilon_{n,1},...,\epsilon_{n,n})$ is stationary and $\mathbb{E} [\epsilon_{n,j}]  = 0 $).
      \item [{\bf (A2)}] There exists a constant $K$ such that for all $j \in \N$ we have
    \begin{align}
    \nonumber %
        \E[\| {\epsilon_{n,j}} \|^{2+\nu}]\leq K, \qquad \E[\| {\epsilon_{n,j}}\|^J]<\infty
    \end{align}
    for some $\nu>0$ and some even integer $J\geq2$.
    \item [{\bf (A3)}] There exists some $\theta \in (0,1]$ with $\theta J>1$ and a non-negative real random variable $M$ with $\E[M^J]<\infty$ such that for any $n \in \N$ and $j=1,...,n$ the inequality
    \begin{align*}
        \vert X_{n,j}(t)-X_{n,j}(t')\vert \leq M |t-t'|^\theta 
    \end{align*}
    holds almost surely. Here $J$ is the same constant as in (A1). 
    \item [{\bf (A4)}] $(\epsilon_{n,j}, j \in \N)$ is $\phi$-mixing with mixing coefficients satisfying for some $\bar \tau \in (1/(2+2\nu),1/2)$ the condition
    \begin{align*}
        \sum_{k=1}^\infty k^{1/(1/2-\bar \tau)}\phi(k)^{1/2}<\infty , \qquad \sum_{k=1}^\infty (k+1)^{J/2-1}\phi(k)^{1/J}<\infty.
    \end{align*}
   % \item [{\bf (B1})] 
   %The sizes satisfy $ \Delta_{\rm thresh}:=\min\limits_{ 1 \leq j \leq m}\| {\mu_j-\mu_{j-1}}\|_\infty \geq \eta>0.$ \pb{Superfluous if we don't let$\mu_j$ be dependent on $n$, I think}
   % \item [{\bf (B2)}]    $\max\limits_{0 \leq j \leq m}\|{\mu_j}\|_\infty <B$ for some positive constant B.
\end{itemize}

Our first result establishes the consistency of the estimator $\hat m$ of the number of change points and of 
the estimators $\hat s_1,...,\hat s_{\hat  m}$
of the location of the  change  points 
defined by Algorithm \ref{basuAlg002}.
The proof  is deferred to the appendix.

\begin{theorem}[consistency of  Algorithm \ref{basuAlg002})]
\label{consistent_num}
Suppose that Assumptions (A1), (A2) and (A4) hold 
and that 
\begin{align}
\nonumber %\label{pb101}
    \frac{\log^{1/2}(n)}{\xi_n}+\frac{\xi_n}{\sqrt{n}}= o(1).
\end{align}
as $n \to \infty$. 
Then 
\begin{align}
\nonumber %\label{pb102}
    \p \Big (
    \hat m=m, \max_{1 \leq i \leq m}\vert \hat s_i - s_i\vert \leq \log(n)/n
    \Big ) \rightarrow 1.  
\end{align}

\end{theorem}

%\pb{
%For the next statement we need to define the following sets
%\begin{align}
%    S_{r}&=\{s_i \in (s_1,...,s_m) | \norm{\mu_{i+1}-\mu_{i}}>\Delta\}\\
%    \hat S_{\rm rel}&=\{\hat s_i \in (\hat s_1,...,\hat s_{\hat m}) |  T_{n,i}>q^*_{1-\alpha} \}
%\end{align} 
\begin{theorem} \label{thm3}
    Suppose  that assumptions (A1) - (A4) hold with $\nu\geq2$ in (A1). If additionally  $L=n^\beta$ with $\beta \in [1/5,2/7]$ such that 
 \begin{align*}
    (\beta(2+\nu)+1)/(2+2\nu)<\bar \tau < 1/2
\end{align*}
in Assumption (A4) we obtain that the estimator of the set of relevant change  points is consistent  in the following sense. 
    \begin{itemize}
    \item [(i)] If  $S_{\rm rel}=\emptyset$, we have  
    \begin{align*}
        \liminf_{n \to \infty}  \p(\hat S_{\rm rel } =\emptyset) \geq 1-\alpha . 
    \end{align*}
  Moreover, if additionally 
  $ \max_{1 \leq i \leq m}\| {\mu_{i+1}-\mu_{i}} \|_\infty <\Delta $ 
  %for some $c>0$ 
  this can be strengthened to  
    \begin{align}
    \label{p7}
    \lim_{n \to \infty}    \p(\hat S_{\rm rel }=\emptyset) = 1
    \end{align}
    \item [(ii)] % Whenever 
%$  \sqrt{n}\left(\min_{1 \leq i \leq m, {s_i \in S_{\rm rel}}}\| {\mu_{i+1}-\mu_{i}}\|_\infty -\Delta\right)\rightarrow \infty$
We have 
    \begin{align} \label{d10a}
      \lim_{n \to \infty}   \p\big (S_{\rm rel } \subset(  \hat S_{\rm rel }+ \log (n)/n)\big)=  1
    \end{align}
    as well as 
    \begin{align}   
    \label{d10b}
    \liminf_{n \to \infty}  \p\big (\hat S_{\rm rel} \subset (S_{\rm rel }+\log(n)/n)\big) \geq 1-\alpha.
    \end{align}
        The latter also can be strengthened to 
    \begin{align}   
    \label{hd10c}
   \lim_{n \to \infty}   \p\big (\hat S_{\rm rel} \subset (S_{\rm rel}+\log (n)/n)\big ) = 1
    \end{align}
    whenever 
$ \max_{1 \leq i \leq m, s_i \notin S_{\rm rel} }\|{\mu_{i+1}-\mu_i}\|_\infty <\Delta$.
%for some $c>0$.
\end{itemize}
\end{theorem}

\begin{remark} \label{rem2} ~~~

{\rm 
\begin{itemize}
    \item[(a)] From a theoretical point of view  our results can be extended to the case, where  the mean functions  on the different segments vary with $n$, that is  $\mu_j=\mu_j^{(n)}$  ($j=1, \ldots  , m$). In this case Theorem \ref{consistent_num} and \ref{thm3}  remain valid under the additional  assumptions:
    \begin{align*}
   &  \min\limits_{n \in \N}\min\limits_{ 1 \leq j \leq m}\| {\mu_j^{(n)}-\mu_{j-1}^{(n)}}\|_\infty \geq \eta>0 ~,    \\
   &  \max\limits_{n \in \N}\max\limits_{0 \leq j \leq m}\|{\mu_j^{(n)}}\|_\infty < \infty ~,  \\
 &     \sqrt{n}\big(\min_{1 \leq i \leq m, {s_i \in S_{\rm rel}}}\| {\mu_{i+1}^{(n)}-\mu_{i}^{(n)}}\|_\infty -\Delta\big)\rightarrow \infty~.
    \end{align*}
    For the additional statements \eqref{p7} and \eqref{hd10c} in
    Theorem \ref{consistent_num} and \ref{thm3}   we further require
    \begin{align*}
     &   \max\limits_{n \in \N}\max_{1 \leq i \leq m}\| {\mu_{i+1}-\mu_{i}} \|_\infty <\Delta~, \\
     &    \max\limits_{n \in \N}\max_{1 \leq i \leq m, s_i \notin S_{\rm rel} }\|{\mu_{i+1}-\mu_i}\|_\infty <\Delta~,
    \end{align*}
    respectively.
       \item[(b)] All results presented in this also hold for sufficiently densely observed functional data.  In this case additional assumptions on the smoothness of the 
    trajectories positively influences the performance of our method as there is a trade off between how densely the data is observed and how well we can interpolate to obtain estimates of the functions.
       \item[(c)]  It is also possible to consider a sequence of nominal levels, say  
    $\alpha=\alpha_n$, in Theorem \ref{thm3}, such that  $\alpha_n$ that converges to $0$ at a polynomial rate. In this case we have instead of \eqref{d10a} and \eqref{d10b}
    $$ 
      \lim_{n \to \infty}   \p\big ( S_{\rm rel}  = \hat S_{\rm rel } \big)=  1 ,f
    $$
    whenever  
    $$
    \frac{\sqrt{n}}{\log(n)}\bigg (\min_{1 \leq i \leq m, s_i \in S_{\rm rel}}\| {\mu_{i+1}^{(n)}-\mu_{i}^{(n)}}\|_\infty -\Delta\bigg )\rightarrow \infty.
    $$
    This  observation is  a consequence of the fact that the statistic 
    $D(\mathcal{E})$ in \eqref{d9} 
    is bounded by  
    $\max_{1 \leq i \leq m} \|{\mathbb{W}_i} \|_\infty $ where the random variables $\mathbb{W}_i$ have the same distribution as the process $\mathbb{W}$ in\eqref{d9}. By Fernique's Theorem   \citep[see][]{Fernique1975} each  random variable $\|{\mathbb{W}_i}\|_\infty$ has subgaussian tails. This   implies at most logarithmic growth of the quantiles of $D(\mathcal{E})$ in $1/\alpha$.
   \item[(d)] We note that the localization rates obtained in Theorem \ref{consistent_num} are minimax optimal up to a logarithmic factor (see the discussion in section 1.4 of \cite{Verzelen2023}). 
    \end{itemize}
    }
\end{remark}

\textbf{Ethics statement: } The  ethics committee at the university of Twente., (Ethical Committee EEMCS (EC-CIS), University of Twente, ref.: RP 2022-20) approved the experimental protocol of this study.

\textbf{Additional resources :} Data may be made available on reasonable request. The source code used in this paper is available on github \href{https://github.com/rupsabasu2020/multiple_relevant_CP}{rupsabasu2020/multiple\_relevant\_CP}.

\textbf{Acknowledgments}
This work  was partially supported by the  
Deutsche Forschungsgemeinschaft: 
 DFG Research unit 5381 {\it Mathematical Statistics in the Information Age}, project number 460867398 and by the  project titled: Modeling functional time series with dynamic factor structures and points of impact",  with project number 511905296. We thank the team of \textit{Sports, Data, and Interaction\footnote{\url{http://www.sports-data-interaction.com}}},  in particular Robbert van Middelaar and Aswin Balasubramaniam for meticulously collecting, pre-processing and providing the data used in this work and Bram Kohlen for helping refine the implementation of \cref{basuAlg002} on python. %The authors would like to thank two referees and the associate editor for  constructive comments, which led to a substantial improvement of  an earlier version of this paper.

\bibliographystyle{apalike}
\bibliography{cppBib}       % Bibliography file (usually '*.bib')

\newpage 

\section{Appendix}
\label{sec6} 
  \def\theequation{6.\arabic{equation}}	
\setcounter{equation}{0}

\section{Supplement}
In this  appendix we represent the  proofs of the theoretical results in Section \ref{sec4}. Recall that 
 $\hat{S} = \{\hat{s}_1, \dots, \hat{s}_{m}\}$ is the set of  estimated change points. 
Throughout  this section $\|{\cdot} \|_2$ 
 always refers to the $L^2$-norm.

\subsection{Proof of Theorem \ref{consistent_num}  (consistency of BINSEG)} \label{sec61}

 The proof follows from Theorem 2.2 in \cite{rice2022consistency}. To apply this result we 
need to verify the Assumption 2 - 5 of this reference. We begin with Assumption 3.
\smallskip

Let $\omega_{j}$ denote  the  modulus of continuity of  the function $\mu_{j}-\mu_{j-1}$ and note that $|\mu_{j}(t)-\mu_{j-1}(t')|\leq \omega_{j}(|t-t'|)$. 
As $\omega_{j}(0)=0$,  $\omega_{j}$ is continuous, increasing and
$\mu_{j}-\mu_{j-1} \not = 0 $, 
there exists a positive constant  $\delta_{j}>0$ such that $\omega_{j}(\delta_{j})=\norm{\mu_{j}-\mu_{j-1}}_\infty/2$. Thus we can find an interval $I \subset [0,1]$ with length at least $\delta_{j}$ such that   $|\mu_{j}(t)-\mu_{j-1}(t')|>\norm{\mu_{j}-\mu_{j-1}}_\infty/2$ for all $t,t' \in I$. As  there are only finitely many of these differences it follows that 
\begin{align*}
    \min_{1 \leq j \leq m }\norm{\mu_{j}-\mu_{j-1}}_2 \geq \eta  := \min_{1 \leq j \leq m}\norm{\mu_{j} -\mu_{j-1}}_\infty \cdot \min_{1 \leq j \leq m}\delta_{j}  ~.
\end{align*}
% By Assumption (A3) it follows that
%$\inf_{n \in \mathbb{N}} \eta_n >0 $
% as we can use the estimate  $ |\mu_{n,j}(t)-\mu_{n,j}(t')|\leq \E[M]|t-t'|^\theta$
% to obtain a positive  lower bound for   $\inf_{n \in \mathbb{N}} \min_{1 \leq j \leq m}\delta_{n,j}$.
Therefore, Assumption 3 in \cite{rice2022consistency} is satisfied.
%\pb{to be precise we have
%\begin{align*}
%    |\mu_j(t)-\mu_j(t')|\leq \E[|X_{s_{j-1}}(t)-X_{s_{j-1}}(t')|]\leq \E[M]|t-t'|^\theta~,
%\end{align*}
%which yields the desired modulus by the triangle inequality.}
Assumption 4 is immediate as  $m$ is fixed. Assumption 5 follows from the fact that $\norm{\mu_j}_\infty<\infty$ in combination with  $\| {\cdot}\|_2 \leq \|{\cdot}\|_\infty$. Finally,  Assumption 2 is a consequence of the following  Lemma.

\begin{lemma}
\label{pl1}
Assume that $(\epsilon_{i})_{i\in \mathbb{N}} $ is a stationary sequence with $\E[\| {\epsilon_i}\|_2]\leq K$ such that
$
    \sum_{n=1}^\infty \phi(n)^{1/2} < \infty .
$
Then we obtain for the mean $\Bar{\epsilon}_k={1 \over k} \sum_{i=1}^k\epsilon_i$.
    \begin{align*}
        \p \big (\max_{1 \leq k \leq n} \sqrt{k}\| {\Bar{\epsilon}_k}\|_2 > x \Big ) \leq C \log(n)x^{-2}
    \end{align*}
    where $C$ depends only on $K$ and $(\phi(n))_{n \in \N}$.
\end{lemma}
\begin{proof}
    Using similar arguments as in the proof of Lemma B.1 in 
    the online supplement to \cite{Aue2017} we obtain
    \begin{align*}
         \p \Big (\max_{1 \leq k \leq n} \sqrt{k}\norm{\Bar{\epsilon_k}}_2 > x \Big ) \leq \sum_{j=1}^{c\log(n)}2^{-(j-1)}x^{-2}\E\Big [\Big (\max_{1 \leq k \leq 2^j} \norm{\sum_{i=1}^k\epsilon_i}_2\Big)^2\Big]
    \end{align*}
    (note that this estimate does not depend on the dependence structure of the random variables).
    Observe that $(\epsilon_i)_{i\in \mathbb{N}} $  $\phi$-mixing implies that $( \| \epsilon_i \|_2)_{i\in \mathbb{N}} $  is $\phi$-mixing with mixing coefficients at most as large as those of $(\epsilon_i)_{i\in \mathbb{N}} $. In particular we may use Theorem 2.1 from \cite{Xuejun2009} to obtain
    \begin{align*}
        \E\Big [\Big (\max_{1 \leq k \leq 2^j}\norm{\sum_{i=1}^k\epsilon_i}_2\Big )^2\Big]\leq 2^jK_0 ,
    \end{align*}
    where $K_0$ depends only on $K$ and $(\phi(n))_{n \in \N}$, this then yields
    \begin{align*}
         \p \Big (\max_{1 \leq k \leq n} \sqrt{k}\| \Bar{\epsilon_k}\|_2> x \Big ) \leq 2cK_0\log(n)x^{-2}~. 
    \end{align*}
\end{proof}

\subsection{Some preliminary steps for the proof of  Theorem \ref{thm3}} \label{sec62} 

In this section we will prove a preliminary result regarding a test for the hypothesis 
that  there exists at least one  \textit{ structural break of relevant size} in the sequence of data $X_{n,1},X_{n,2} , \ldots , X_{n,n}$, i.e  we consider the hypotheses
\begin{align} 
    H_0: \max_{1 \leq k \leq m}\norm{\mu_{k}-\mu_{k-1}}_\infty \leq \Delta \quad \text{vs. }  H_1: \max_{1 \leq k \leq m}\norm{\mu_{k}-\mu_{k-1}}_\infty  > \Delta .
    \label{mult_null_hyp}
\end{align} 
In fact it turns out that the decision rule 
\begin{align}
\nonumber % \label{boottest}
    \hat T_n>q^*_{1-\alpha}
\end{align}
defines a  consistent and asymptotic level $\alpha$ test. 
Besides of being of interest on its own, this 
 result is the basis for establishing  the consistency of  Algorithm \ref{basuAlg001} with respect to the estimation of the relevant change points
(see Theorem \ref{thm3} below).
\begin{theorem}
\label{thm2}
Let $q^*_{1-\alpha}$ be the $(1-\alpha)$-quantile obtained in the course of Algorithm \ref{basuAlg001}, and suppose  that assumptions (A1) - (A4)  hold with $\nu\geq2$ in (A1). If additionally 
 $L=n^\beta$ with $\beta \in [1/5,2/7]$ such that 
 \begin{align*}
    (\beta(2+\nu)+1)/(2+2\nu)<\bar \tau < 1/2
\end{align*}
in Assumption (A4), the following statements hold true. 
\begin{itemize}
    \item [(i)] Under the null hypothesis of no relevant change point we have  
    \begin{align*}
         \limsup_{n \to \infty}\p(\hat T_n>q^*_{1-\alpha})\leq \alpha~,
    \end{align*}
    with equality whenever $\| {\mu_{i}-\mu_{i-1}} \|_\infty =\Delta$ for all $i=1, \ldots , m$.
    \item [(ii)] Whenever  $\left(\max_{1 \leq i \leq m}\norm{\mu_{i}-\mu_{i-1}}-\Delta\right)>0$ we have
    \begin{align*}
        \lim_{n \to \infty}\p(\hat T_n>q_{1-\alpha}^{*})=1.
    \end{align*}
\end{itemize}
\end{theorem}

The proof of Theorem \ref{thm2} consists of three steps (which will be proved below). We first derive the asymptotic distribution of the statistic $\hat D_n$ defined in \eqref{d9}.

\begin{theorem}[weak convergence of $\hat D_n$]
\label{thm4}
 Under the assumptions of Theroem \ref{thm2} we have
$ \hat D_n \overset{\mathcal{D}}{\rightarrow} D(\mathcal{E}), $ where $\hat D_n$ and $D(\mathcal{E})$ are defined in \eqref{d9}.
\end{theorem}

Next  we derive a (consistent) and asymptotic level $\alpha$ for the hypotheses  \eqref{mult_null_hyp} of no relevant change point using the statistic $\hat T_n$ defined  in  \eqref{basuT00} and the $(1-\alpha)$-quantile of the random variable $D(\mathcal{E})$ in \eqref{d9} (note that this test is not feasible as the quantile depends in a complicated manner on the extremal sets and the dependence structure of the process).

\begin{theorem}
\label{thm5}
Let $q_{1-\alpha}$ denote  the $(1-\alpha)$ quantile of $D(\mathcal{E})$. Under the Assumptions of Theorem \ref{thm2} we have  
\begin{itemize}
    \item [(i)] Under the null hypothesis  it holds  that
    $
        \limsup_{n \to \infty }\p(\hat T_n>q_{1-\alpha})\leq \alpha
    $
    with equality whenever $\norm{\mu_{i}-\mu_{i-1}}_\infty=\Delta$ for all $i=1, \ldots , m $.
    \item [(ii)] 
    $
        \lim_{n\to \infty} \p(\hat T_n>q_{1-\alpha})=1
    $ whenever  $\left(\max_{1 \leq i \leq m}\norm{\mu_{i}-\mu_{i-1}}_\infty-\Delta\right)>0$.
\end{itemize}
\end{theorem}

Finally we show that the multiplier bootstrap proposed in Algorithm \ref{basuAlg001} is produces asymptotically the same distribution as the statistic $\hat D_n$ in Theorem \ref{thm4}. For this purpose we show the following result.

 \begin{theorem} \label{thm6} 
     Under Assumptions of Theorem \eqref{thm2} we have
     \begin{align}
     \nonumber %
         (\hat D_n, \hat T_n^{(1)}, \dots, \hat T_n^{(R)}) \overset{\mathcal{D}}{\rightarrow} (D (\mathcal{E}), D^{(1)} (\mathcal{E}), \dots, D^{(R)} (\mathcal{E}) )~,
     \end{align}
     where $D (\mathcal{E}), D^{(1)} (\mathcal{E}), \dots, D^{(R)} (\mathcal{E}) $  are independent copies of $D (\mathcal{E})$. 
 \end{theorem}

 The assertion of Theorem \ref{thm2} now follows combining Theorem \ref{thm4} - \ref{thm6}.

\subsubsection{Proof of Theorem \ref{thm4}}
\label{sec621}

 Throughout this section the symbol 
$\rightsquigarrow$
denotes weak convergence  in  the space  of
bounded functions on a compact rectangular region, which will always be clear from the context. We now have the following result. 
\begin{theorem}\label{thm7}
\begin{align*}
    \left\{\sqrt{n_i}\left( \hat U_{\hat k_{i-1},\hat k_{i+1}}(\hat h^{-1}(s),t)-(s \land \hat h_i(s_i) -s\hat h_i(s_i))(\mu_{i+1}(t)-\mu_{i}(t))\right) \right\}_{(s,t) \in [0,1]^2 } \rightsquigarrow 
 \{ \mathbb{W}(s,t) \}_{(s,t) \in [0,1]^2 },
\end{align*}
where $\mathbb{W}$ is a centered Gaussian process on $[0,1]^2$ with covariance structure 
\begin{align}
\nonumber %\label{p10}
    \text{\rm Cov}(\mathbb{W}(s,t),\mathbb{W}(s',t'))=(s\land s' -ss')  \sum_{i=-\infty}^{\infty} \text{\rm Cov}(\epsilon_{n,0}(t),\epsilon_{n,i}(t))   ~.
\end{align}
 The latter definition does not depend on the choice of $j$ by the row-wise stationarity of the error process array.
The convergence also holds jointly with respect to $1 \leq i \leq m$.
\end{theorem}

\begin{proof}

Let $k_i=\floor{ns_i}$ $(i=0,...,m+1)$ and consider the process
\begin{align}
 \nonumber %   \label{d13}
 \hat U_{k_{i-1},k_{i+1}} (h^{-1}(s),t)  =& \frac{1}{n_i} \Big ( \sum_{j = k_{i-1}+1}^{\lfloor h^{-1}_i(s)n \rfloor} X_{n,j} (t) +  n_i \Big  (h^{-1}(s)- \frac{\lfloor h^{-1}(s)n_i \rfloor }{n_i} \Big  ) X_{n,k_{i+1}} (t) \\
 &\quad- s\sum_{j=k_{i-1}+1}^{ k_{i+1}} X_{n,j} (t)\Big ),  
 \nonumber %
\end{align}
on $[0,1]^2$, where $n_i=n(s_{i+1}-s_{i-1})$. By the proof of Theorem 4.1 from \cite{dette2020functional} it follows that the statement of Theorem \ref{thm7} holds if $\hat U_{\hat k_{i-1},\hat k_{i+1}}(\hat h^{-1}(s),t)$ is replaced by $  \hat U_{k_{i-1},k_{i+1}}(h^{-1}(s),t)$.  
By equation \eqref{pb1} from the proof of Lemma \ref{pl2} the difference  
$$
\hat U_{\hat k_{i-1},\hat k_{i+1}}(\hat h^{-1}(s),t) -    \hat U_{k_{i-1},k_{i+1}}(h^{-1}(s),t)
$$
is  of order $o_\p(1/\sqrt{n})$ uniformly with respect to  $s,t \in [0,1]$. Therefore the first assertion follows.
The second assertion is a  consequence of the weak convergence of the finite dimensional distributions of the vector $( \hat U_{0,k_2}, \ldots ,  \hat U_{k_{m-1},1})^\top $, which can be verified by tedious but straightforward calculations (note that equicontinuity is obvious as each component is equicontinuous).
\end{proof}

To continue with the proof of Theorem \ref{thm4}  define 
\begin{align*}
 D_n&:= \max_{1 \leq i \leq m}    \big \{ 
\sqrt{n_i}\big ( M_{n,i}-h_i(s_i)(1-h_i(s_i))\|{\mu_{i}-\mu_{i-1}} \|_\infty \big ) \big \} 
%\\
%\sqrt{n_2}\left(\hat M_{n,2}-\hat h_i(s_2)(1-\hat h_i(s_2))\| {\mu_3-\mu_{2}}\| _\infty)\right)\\
%\vdots \\
%\sqrt{n_m}\left(\hat M_{n,m}-\hat h_i(s_m)(1-\hat h_i(s_m))\norm{\mu_{m+1}-\mu_m)}_\infty \right)
%\\& \overset{\mathcal{D}}{\rightarrow} 
%D(\mathcal{E}):=\max_{1 \leq i \leq m}\max\left\{ \sup_{t\in \mathcal{E}_i^+} \mathbb{W} (h_i(s_i), t), \sup_{t\in \mathcal{E}_i^-} -\mathbb{W} (h_i(s_i), t)   \right\}
\end{align*}
where $ M_{n,i}$ 
is obtained from $\hat M_{n,i}$ by replacing the estimates  $\hat s_i$
by $s_i$. Consequently,  $D_n$ is the analog of  $ \hat D_n$  in \eqref{d9}, where  the estimates $\hat s_i$ have been replaced by the unknown change points. As  each $  M_{n,i}$ is a supremum taken over a fixed interval we can apply the delta method for directionally Hadamard differentiable functions \citep[see][Corollary 2.3]{carcamo2020directional} and  obtain from  Theorem \ref{thm7}  that
\begin{align*}
    \big \{ 
& \sqrt{n_i}\big ( M_{n,i}-h_i(s_i)(1- h_i(s_i))\|{\mu_{i}-\mu_{i-1}}\|_\infty)\big ) \big \}_{i=1, \ldots ,m}  \\
& ~~~~~~~~~~~~~~~~~~~~~~~~~~~~~~~~~~~~~~~~~
\overset{\mathcal{D}}{\rightarrow} 
\max\Big \{ \sup_{t\in \mathcal{E}_i^+} \mathbb{W} (h_i( s_i), t), \sup_{t\in \mathcal{E}_i^-} -\mathbb{W} (h_i( s_i), t)   \Big \}.
\end{align*}
By  the continuous mapping theorem we get $  D_n \overset{\mathcal{D}}{\rightarrow} D(\mathcal{E}).  $
The proof of Theorem \ref{thm4} is  completed by  an application of 
    Lemma \ref{pl2}, which shows that we  can  replace each   $ M_{n,i}$  by $\hat M_{n,i}$, yielding $ \hat  D_n \overset{\mathcal{D}}{\rightarrow} D(\mathcal{E}).  $

\subsubsection{Proof of Theorem \ref{thm5}} \label{sec622} 

    The first part follows directly  from Theorem \ref{thm4} using  the inequality $\hat D_n \leq \hat T_n$ which holds under the null hypothesis.   For the second part we first observe that $\| {\mu_{i}-\mu_{i-1}}\|_\infty> \Delta$ for at least one $i \in \{1, \ldots , m\}$. We denote the index with the largest jump size (with respect to the sup-norm) by $i_0$. Recalling the definition of  $\hat T_{n,i}$ in \eqref{d4} and  $D_{n,i}$ in \eqref{d4a} we have     
    \begin{align}
\nonumber %    \label{pb100}
        \hat T_n \geq&  \hat T_{n,i_0}  =  D_{n,i_0} +  \hat h_i(\hat{s}_{i_0})(1-\hat h_i(\hat{ s}_{i_0}))
        \big ( \|{\mu_{i_0}-\mu_{i_0-1}}\|_\infty - \Delta \big ).
%     \\    &=\sqrt{\hat n_{i_0}}\left(\hat M_{n,i_0}-\hat h_i(\hat{s}_{i_0})(1-\hat h_i(\hat{ s}_{i_0}))\Delta\pm \hat h_i(\hat{s}_{i_0})(1-\hat h_i(\hat{ s}_{i_0})\norm{\mu_{i_0}-\mu_{i_0-1}}_\infty \right)
    \end{align}
    By Theorem \ref{thm4} the sequence $(D_{n,i_0} )_{n\in \mathbb{N}}$ is tight, the right hand side of the preceding equation thus diverges to positive infinity. As the quantile $q_{1-\alpha}$ is a bounded quantity the theorem follows.

\subsubsection{Proof of Theorem \ref{thm6}} \label{sec623} 

 Recall the definition of $\hat U_{l,r}$ in \eqref{hd1} and define  $M_{n,i}=\|{ \hat U_{k_{i-1},k_{i+1}}(h^{-1}_i(s),t)}\|_{\infty}$ and $\hat M_{n,i}=\| {\hat U_{\hat k_{i-1},\hat k_{i+1}}(\hat h^{-1}(s),t)}\|_{\infty}$. The proof follows directly from the following two statements.

\begin{lemma}
\label{pl2}
Under the assumptions of Theorem \ref{thm6} we have 
$ \sqrt{n_i}(M_{n,i}-\hat M_{n,i})=o_\p(1)
$.
\end{lemma}
\begin{proof}
\label{proof_consistent_locs}
 It follows from  Theorem \ref{consistent_num} that $|\hat s_i-s_i|\leq \log(n)/n$ holds with high probability.
 We have 
 \begin{align*}
  &  \sqrt{n_i} ( \hat U_{k_{i-1},k_{i+1}}(h^{-1}_i(s),t)-{\hat U_{\hat k_{i-1},\hat k_{i+1}}(\hat h^{-1}(\tilde s),t)}) \\
 & \quad\quad\quad\quad\quad   = \frac{1}{\sqrt{n}}\bigg (\sum_{j=\lfloor sn \rfloor+1}^{\lfloor \tilde sn \rfloor}X_{n,j}(t) +n_i(s-\frac{\lfloor sn_i \rfloor}{n_i})X_{n,s_{i+1}}(t)
    %-n(\tilde s-\frac{\lfloor \tilde sn \rfloor}{n})X_{n,\hat s_{i+1}}(t)
    \\
    & \quad\quad\quad\quad\quad  -\hat n_i(\tilde s-\frac{\lfloor \tilde s\hat n_i \rfloor}{\hat n_i})X_{n,\hat s_{i+1}}(t) -(h_i(s)-\hat h_i(\tilde s))\sum_{j=k_{i-1}+1}^{k_{i+1}} X_{n,j}(t) \Bigg)+o_\p(1)\\
     & \quad\quad\quad\quad\quad =\frac{1}{\sqrt{n_i}}\sum_{j=\lfloor sn \rfloor}^{\lfloor \tilde sn \rfloor}X_{n,j}(t)+o_\p(1) =\frac{1}{\sqrt{n_i}}\sum_{j=\lfloor sn \rfloor}^{\lfloor \tilde sn \rfloor}\epsilon_{n,j}(t)+o_\p(1)
 \end{align*}
 uniformly with respect to $t \in [0,1] $ and all 
$s<\tilde s$ with $\sqrt{n}\vert s-\tilde s\vert \leq \log(n)/\sqrt{n}$.
 An application of Markov's inequality then yields
 \begin{align}
 \label{pb1}
     \p\Big (\sqrt{n_i} \sup_{t \in [0,1]} \sup_{ |s-\tilde s |\leq \frac{\log(n)}{{n}}} \big |\hat U_{k_{i-1},k_{i+1}}(h^{-1}_i(s),t)-\hat U_{\hat k_{i-1},\hat k_{i+1}}(\hat h^{-1}(\tilde s),t)\big | > x \Big )  \leq 
    {2K  \frac{ \log n}{x}} +o(1) =o(1)    
 \end{align}
 Finally we observe that 
 \begin{align*}
      &\sqrt{n}\Big (\big \| \hat  U_{k_{i-1},k_{i+1}}(h^{-1}_i(s),t) \big \|_{\infty}-\norm{\hat U_{\hat k_{i-1},\hat k_{i+1}}(\hat h^{-1}(s),t)}_{\infty}\Big ) \\
      &\leq \sqrt{n} \sup_{t \in [0,1]} \sup_{ \sqrt{n}|s-\tilde s|\leq \frac{\log(n)}{\sqrt{n}}} | \hat  U_{k_{i-1},k_{i+1}}(h^{-1}_i(s),t)-\hat U_{\hat k_{i-1},\hat k_{i+1}}(\hat h^{-1}(\tilde s),t) | 
 \end{align*}
 on the set where $\sqrt{n}|s_i-\hat s_i|\leq \frac{\log(n)}{\sqrt{n}}$ for $i=1,...,m$. This set has probability tending to one which in combination with \eqref{pb1} yields the assertion.

\end{proof}

 \begin{lemma}
     \label{thm8}
      Define
     \begin{align*}
         \widehat{W}_{i,n}(s,t):=\hat U_{k_{i-1},k_{i+1}}(s,t)-(h_i(s\land s_i)-h(s)h_i(s_i))(\mu_i(t)-\mu_{i-1}(t))
     \end{align*}
     on the rectangle $[s_{i_1},s_{i+1}]\times [0,1]$.
     Under the assumptions of Theorem \ref{thm6} we have 
     \begin{align}
     \label{pb30}
      \widehat{W}_{n} := 
     \begin{pmatrix}
                  \widehat{W}_{1,n}, \widehat{W}_{1,n}^{(1)}, \dots, \widehat{W}_{1,n}^{(R)} \\
                   \widehat{W}_{2,n}, \widehat{W}_{2,n}^{(1)}, \dots, \widehat{W}_{2,n}^{(R)}\\ 
                   \vdots \\
                    \widehat{W}_{m,n}, \widehat{W}_{m,n}^{(1)}, \dots, \widehat{W}_{m,n}^{(R)}  
     \end{pmatrix}\rightsquigarrow W:=
     \begin{pmatrix}
          W_{1}, W_{1}^{(1)}, \dots, W_{1}^{(R)}\\
          W_{2}, W_{2}^{(1)}, \dots, W_{2}^{(R)}\\
          \vdots \\
          W_{m}, W_{m}^{(1)}, \dots, W_{m}^{(R)}  
     \end{pmatrix}~, 
     \end{align} 
     where $\{W_i^{(1)} | 1 \leq i \leq m \}, \ldots , \{W_i^{(R)} | 1 \leq i \leq m \} $ are $R$ independent copies of $\{W_i | 1 \leq i \leq m \}$, which are defined by
     \begin{align}
 \nonumber %    \label{pb31}
    \left\{W_i(s,t)\right\}_{(s,t)  \in [s_{i-1},s_{i+1}] \times [0,1]} \stackrel{\mathcal{D}} \sim \left\{ \mathbb{W} (h_i(s),t)\right\}_{(s,t) \in [s_{i-1},s_{i+1}] \times [0,1]}~.
\end{align} 
\end{lemma}
 \begin{proof}
 Define 
 \begin{align} \label{d12}
    \widetilde{W}_{n} (s,t) = \widehat{W}_{n}  (h_i(s),t) \mbox{  and  }  
 \widetilde{W} (s,t) =  {W} (\hat h_i(s),t)~.
 \end{align}
 Using 
      similar arguments as in the proof of Theorem \ref{thm7} (using the proof of Theorem 4.3 instead of 4.1 from \cite{dette2020functional}) we  obtain that each row  of   $\widetilde{W}_{n} $ converges weakly to the corresponding row of 
      $\widetilde W$.  A  straightforward but tedious argument shows that 
    the finite dimensional distributions over both rows and columns $\widetilde{W}_{n} $
    converge weakly to the corresponding finite dimensional distribution of $\widetilde W$.
   As marginal tightness is equivalent to tightness of the vector 
   we  therefore obtain $\widetilde{W}_{n}    \rightsquigarrow 
  \widetilde{W} $. The weak convergence in 
     \eqref{pb30} now follows by inverting the operation in \eqref{d12}.
 \end{proof}

We now turn to the proof of Theorem \ref{thm6}. For this purpose we note that it follows by the same arguments as in  \cite{dette2020functional} that 
$
        d_H (\hat{\mathcal{E}}_{i}^{\pm}, \mathcal{E}_i) \stackrel{\mathbb{P}} \rightarrow 0$ ($i=1, \ldots , m$),  where $d_H$ denotes the Hausdorff distance.
The assertion of Theorem \ref{thm6} now follows 
 from Lemma \ref{pl2} and   \ref{thm8}  in combination with Lemma B.3 in the same reference.

\subsection{Proof of Theorem \ref{thm3}}
\label{sec63}

For a proof of (i) note that when $\max_{1 \leq i \leq m}\norm{\mu_{i}-\mu_{i-1}}_\infty\leq \Delta$ we have by Theorem \ref{thm2} that 
    \begin{align*}
   \liminf_{n \to \infty}    \p \big (  \{\hat S_{\rm rel} =\emptyset \} \big ) = 1 -   \limsup_{n \to \infty} \p(\hat T_n >  q_{1-\alpha}^*)\geq 1-\alpha . 
    \end{align*}
\smallskip 

For  a proof of part (ii) we begin by showing \eqref{d10a}. For this purpose we note that
    \begin{align*}
     &    \{S_{\rm rel} \subset (\hat S_{\rm rel}+\log(n)/n)\}  \\ 
     & ~~~~~~~~~~~
     \supset  \{ \hat T_{n,i} > q^*_{1-\alpha}   \forall \,  i \text{ s.t. } s_i \in S_{\rm rel}\} \cap \{\hat m=m, \max_{1 \leq i \leq m }|\hat s_i-s_i|\leq \log(n)/n \}
    \end{align*}
 and recall by the arguments in the proof of Theorem \ref{thm5} that $\hat T_{n,i}>q^*_{1-\alpha}$ holds with high probability for any $i$ with $(\norm{\mu_{i}-\mu_{i-1}}_\infty-\Delta)>0$. 
 By Theorem \ref{consistent_num} it follows  that $\{\hat m=m, \max_{1 \leq i \leq m }|\hat s_i-s_i|\leq \log(n)/n \}$ holds with high probability which  yields the desired assertion.

    For equation \eqref{d10b} we first note that
    \begin{align*}
     \{\hat S_{\rm rel} \subset (S_{\rm rel }+\log(n)/n)\} \supset \{ \hat T_{n,i}\leq q^*_{1-\alpha}, i \text{ s.t. } s_i \notin S_{\rm rel}\} \cap \{S_r \subset (\hat S_{\rm rel}+\log(n)/n)\}
    \end{align*}
    and that 
    \begin{align}
    \label{p105}
        \p(\max_{i \notin S_{\rm rel}}\hat T_{n,i}>q_{1-\alpha}^*)\geq 1-\alpha
    \end{align}
   by Theorem \ref{thm2}. 
    For equation \eqref{hd10c} note that the inequality \eqref{p105} can be improved to $1-o(1)$ whenever 
    \begin{align*}
        \max_{1 \leq i \leq m, s_i \notin S_{\rm rel} }\|\mu_{i}-\mu_{i-1}\|_\infty <\Delta
    \end{align*} 
    observing the inequality
    \begin{align}
    \nonumber %
        \hat T_n\leq \hat D_n + \sqrt{n}(\max_{1 \leq i \leq m, s_i \notin S_{\rm rel} }\|\mu_{i}-\mu_{i-1}\|_\infty-\Delta)
    \end{align}
    and  that $ \hat D_n$  is  tight.

\end{document}